\documentclass[12pt]{amsart}

\usepackage{a4wide}
\usepackage[normalem]{ulem}
\usepackage{amsmath}
\usepackage{amssymb,amsthm}
\usepackage{graphicx}

\usepackage{enumerate}

\usepackage{tikz}

\usepackage{hyperref}

\usepackage{color}
\definecolor{darkblue}{rgb}{0,0,0.8}

\newtheorem{propo}{Proposition}[section]

\newtheorem{lemma}[propo]{Lemma}
\newtheorem{corol}[propo]{Corollary}

\newtheorem{theo}[propo]{Theorem}

\newtheorem{prop}[propo]{Proposition}

\newcommand\Cay{{\rm Cay}}

\newcommand\PSL{{\rm PSL}}

\newcommand\PG{{\rm PG}}
\newcommand\PGL{{\rm PGL}}
\newcommand\AGL{{\rm AGL}}
\newcommand\PGammaL{{\rm P}\Gamma{\rm L}}
\newcommand\GL{{\rm GL}}
\newcommand\GammaL{\Gamma{\rm L}}

\newcommand\SigmaL{\Sigma{\rm L}}
\newcommand\PSigmaL{{\rm P}\Sigma{\rm L}}
\newcommand\GF{{\rm GF}}

\newcommand\Sym{{\rm Sym}}

\newcommand\Aut{{\rm Aut}}

\newcommand\K{{\rm K}}

\newcommand\Ha{{\rm H}}

\begin{document}
\title{Arc-transitive bicirculants }

\thanks{The second author was supported by the Australian Research Council grant DP150101066 while  the third author is supported by NSFC (11661039,12061034) and NSF of Jiangxi (2018ACB21001, 20192ACBL21007,GJJ190273)}

\author[A. Devillers]{Alice Devillers}
\address{Alice Devillers\\Centre for the Mathematics of Symmetry and Computation\\ Department of Mathematics and Statistics\\
The University of Western Australia\\
Crawley, WA 6009, Australia} \email{alice.devillers@uwa.edu.au}
\author[M. Giudici]{Michael Giudici}
\address{Michael Giudici\\Centre for the Mathematics of Symmetry and Computation\\  Department of Mathematics and Statistics\\
The University of Western Australia\\
Crawley, WA 6009, Australia} \email{michael.giudici@uwa.edu.au}
\author[W. Jin]{Wei Jin}
\address{Wei Jin\\School of Statistics\\
Jiangxi University of Finance and Economics\\
 Nanchang, Jiangxi, 330013, P.R.China}
\address{School of Mathematics and Statistics\\
Central South University\\
Changsha, Hunan, 410075, P.R.China}
 \email{jinweipei82@163.com}


\begin{abstract}
In this paper,  we characterise  the family of finite arc-transitive bicirculants.
We show that every finite  arc-transitive  bicirculant
is a normal $r$-cover of an arc-transitive graph that lies in one of  eight infinite families or is one of seven sporadic arc-transitive graphs.
Moreover, each of these ``basic'' graphs  is either an arc-transitive bicirculant or an arc-transitive circulant, and  each graph in the latter case has an arc-transitive bicirculant normal $r$-cover for some integer $r$.
\end{abstract}

\maketitle

\vspace{2mm}

 \hspace{-17pt}{\bf Keywords:}  bicirculant, arc-transitive graph, permutation group.

 \hspace{-17pt}{\bf Math. Subj. Class.:} 05E18; 20B25

\section{Introduction}

A graph with $2n$ vertices   is called a \emph{bicirculant} if it admits an automorphism $g$ of order $n$ with exactly two cycles of
length $n$. The class of bicirculants includes Cayley graphs of dihedral groups, the Generalised Petersen graphs \cite{FGW-1971}, Rose-Window graphs \cite{Wilson} and Taba\v{c}jn graphs \cite{AHKOS-2015}.
The last  three families are particularly nice as the two orbits of $\langle g\rangle$ are cycles in the graph.
Arc-transitive bicirculants have been classified for small  valencies, for example valency three \cite{FGW-1971,MP-2000,P-2007}, valency four \cite{KKM-2010,KKM,KKMW-2012} and valency five  \cite{AHK-2015,AHKOS-2015}. Moreover, bicirculants of valency 6 were recently investigated \cite{JMSV}, while the automorphisms of bicirculants on $2p$ vertices for $p$ a prime are well understood \cite{MMSF-2007}.

We call a graph with $n$ vertices  a {\it circulant} if it has an automorphism $g$ that is an  $n$-cycle.
Let $\Gamma $ be a connected arc-transitive circulant which is not a complete graph. Then
Kov\'acs \cite{Kovacs-2004} and Li \cite{LCH-circulant-2005} proved that
either $\langle g\rangle $ is a normal subgroup of $ \Aut(\Gamma)$, or $\Gamma $ is the lexicographic product of
an  arc-transitive circulant   and  an empty graph,
or $\Gamma $ is obtained by  the lexicographic product of
an  arc-transitive circulant $\Sigma$  and  an empty graph and then removing some copies of $\Sigma$.
In this paper, we continue the study of arc-transitive bicirculants and give a characterisation for these graphs  which is of a similar style to  the result of Kov\'acs  and Li.

Let $\Gamma$ be a $G$-vertex-transitive graph. Let $\mathcal{B}=\{B_1,\ldots,B_m\}$ be a $G$-invariant partition of the vertex set, that is, for each $B_i$ and each $g\in G$, either $B_i^g\cap B_i= \varnothing $ or $B_i^g=B_i$.
Then  the \emph{quotient graph} $\Gamma_{\mathcal{B}}$
of $\Gamma$ induced on $\mathcal{B}$ is the graph with vertex set $\mathcal{B}$, and
$B_i,B_j$ are adjacent if there exist $x\in B_i$ and $y\in B_j$ such that $x,y$ are adjacent in $\Gamma$.
The graph  $\Gamma$ is said to be an
\emph{$r$-cover} of $\Gamma_{\mathcal{B}}$ if for each edge $\{B_i,B_j\}$ of
$\Gamma_{\mathcal{B}}$ and each $v\in B_i$, we have $|\Gamma(v)\cap B_j|=r$. In particular, if $r=1$, then
$\Gamma$ is called a \emph{cover} of $\Gamma_{\mathcal{B}}$.
Moreover, if $\mathcal{B}$ is the set of orbits of an intransitive normal subgroup of $G$, then
$\Gamma$ is called a \emph{normal $r$-cover} of $\Gamma_{\mathcal{B}}$.

Our main theorem  characterises the family of  arc-transitive bicirculants. The definitions of the graphs arising in Theorem  \ref{bicirculant-reduction-th2} will be given in Subsection 2.1.

\begin{theo}\label{bicirculant-reduction-th2}
Every finite connected arc-transitive  bicirculant which is not equal to $\K_2$
is a normal $r$-cover of
one of the following   graphs:
 \begin{enumerate}[{\rm (a)}]
\item    $\K_{n,n}$, $n\geqslant 2$;
\item   $\K_n$, $ \K_{n[2]}$,  $\K_{n,n}-n\K_2$, $n\geqslant 3$;
\item   $G(2p,e)$, where $p$ is a prime and  $e>1$ divides $p-1$;
\item   $B(\PG(n,q))$, $B'(\PG(n,q))$ where $q$ is a prime power and $n\geqslant 2$;
\item  $\Cay(p,e)$,  where  $p$ is a prime and $e$ is an even integer dividing  $p-1$;
\item   $B'(H(11))$;
\item Petersen graph, $\Ha(2,4)$, Clebsch graph, and their  complements.
\end{enumerate}
Moreover, each of the graphs in $(a)-(g)$ is either an arc-transitive bicirculant  or an arc-transitive circulant  that has
an   arc-transitive bicirculant normal $r$-cover.
\end{theo}

Lemma \ref{basiccirc-1} determines exactly which graphs listed in Theorem \ref{bicirculant-reduction-th2} are circulants and Lemma \ref{basiccirc-3} determines which ones are bicirculants. Note that a graph can be both a circulant and a bicirculant.

A transitive permutation group $G\leqslant \Sym(\Omega)$ is said to be \emph{quasiprimitive}, if every non-trivial normal subgroup of $G$ is transitive on $\Omega$, while
$G$ is said to be \emph{bi-quasiprimitive} if every non-trivial normal subgroup of $G$ has at most two orbits on $\Omega$ and there exists one which has exactly two orbits on $\Omega$.
Quasiprimitivity  is a generalisation
of primitivity as every normal subgroup of a primitive group is transitive,
but there exist quasiprimitive groups which are not primitive. For more information about quasiprimitive and bi-quasiprimitive permutation groups, refer to \cite{Praeger-1993-onanscott,Praeger-2,Praeger-2003-biq}.

The proof of Theorem  \ref{bicirculant-reduction-th2} proceeds as follows.
Let $\Gamma$ be a connected $G$-arc-transitive  bicirculant  with at least three vertices (otherwise $\Gamma=\K_2$).
Let $N$ be a  normal subgroup of $G$ maximal with respect to having  at least $3$ orbits.
It is proved in Theorem 3.3 that:
$\Gamma$ is an $r$-cover of the  quotient graph induced by $N$  whose automorphism group contains  $G/N$,
this  quotient graph  is $G/N$-arc-transitive and  is either a circulant or a bicirculant, and $G/N$ is quasiprimitive or bi-quasiprimitive on its vertex set.
The families of  vertex quasiprimitive and  bi-quasiprimitive arc-transitive circulants  are  obtained  in   Proposition \ref{circ-biqp-th1}.
Then we determine precisely the  vertex quasiprimitive arc-transitive bicirculants in  Proposition \ref{bicirculant-quasiprimitive-th1}, and
a key part of the proof is M\"{u}ller's classification of primitive groups containing a cyclic subgroup with two orbits.
The vertex bi-quasiprimitive arc-transitive bicirculants are given in  Propositions \ref{bicirc-biqp-lem2} and \ref{bicir-biqp-2}.
Note that the graphs in Theorem  \ref{bicirculant-reduction-th2} do not necessarily  have an arc-transitive quasiprimitive or  bi-quasiprimitive
group of automorphisms for all values of the parameters.
Each of the graphs in $(a)-(g)$ is either an arc-transitive bicirculant or an arc-transitive circulant by Lemmas \ref{basiccirc-1} and \ref{basiccirc-3}.
Moreover, the ones that are circulants have a normal cover that is an arc-transitive bicirculant by Lemma \ref{circ-sdc}.

Finally, in Section 6,  we relate Theorem \ref{bicirculant-reduction-th2} to the classifications of arc-transitive bicirculants of valencies 3, 4 and 5.

\section{Preliminaries}

In this section, we give some definitions about groups and graphs and also prove some  results which will be used in the
following discussion.

\subsection{Groups and graphs}

All graphs in this paper are finite, simple, connected and undirected. For a graph $\Gamma$, we use  $V(\Gamma)$ and
$\Aut(\Gamma)$ to denote its \emph{vertex set}  and
\emph{automorphism group}, respectively. The size of the vertex set of a graph is said to be the \emph{order} of the graph. For the group theoretic terminology not defined here we refer the reader to \cite{Cameron-1,DM-1,Wielandt-book}.

We denote by $\mathbb{Z}_n$ the cyclic group of order $n$.
Let $G$ be a permutation group on a set $\Omega$ and $\alpha \in \Omega$. Denote by $G_\alpha$ the stabiliser of  $\alpha$ in $G$, that is, the subgroup of $G$ fixing the point $\alpha$. We say that $G$ is \emph{semiregular} on
$\Omega$ if $G_\alpha=1$ for every $\alpha\in \Omega$ and \emph{regular} if $G$ is transitive and semiregular.
An \emph{arc} of  a graph  is an ordered pair of adjacent vertices.  A graph $\Gamma$ is said to be  \emph{$G$-vertex-transitive} or \emph{$G$-arc-transitive} if  $G\leqslant \Aut(\Gamma)$ is
transitive on  the  set of vertices or on  the  set of arcs, respectively. Moreover, if $G= \Aut(\Gamma)$, then we drop the prefix ``$G$-" in the definitions.

Let  $G$ be a transitive permutation group on a set  $\Omega$ and
let $\mathcal{B}$ be a $G$-invariant partition of $\Omega$. Then $G$ induces a transitive permutation group on $\mathcal{B}$, denoted by $G^{\mathcal{B}}$. If the only possibilities for $\mathcal{B}$ are the partition into one part, or the partition into singletons then $G$ is called \emph{primitive}. The \emph{kernel} of $G$ on $\mathcal{B}$ is the normal subgroup of $G$ consisting of all elements that fix setwise each $B\in\mathcal{B}$. We call $\mathcal{B}$ \emph{maximal} if
$G^{\mathcal{B}}$ is primitive on  $ \mathcal{B}$.
Let $B$ be a non-empty
subset of $\Omega$. Then $B$ is called a \emph{block} of
$G$ if, for any $g\in G$, either $B^g=B$ or
$B^g\cap B =\varnothing$.
If $N$ is an intransitive normal subgroup of $G$, then each $N$-orbit is a block of $G$, and the
set of $N$-orbits $\mathcal{B}$   forms a
$G$-invariant partition of $\Omega$.
Let $\Gamma$ be a $G$-vertex-transitive graph and  let $\mathcal{B}$ be  the set of $N$-orbits on $V(\Gamma)$ for some normal subgroup $N$ of $G$,
then we denote $\Gamma_{\mathcal{B}}$ by $\Gamma_N$ and call $\Gamma_N$ a \emph{normal quotient graph}.

For a finite group $T$, and a subset $S$ of $T$ such that $1\notin
S$ and $S=S^{-1}$, the \emph{Cayley graph} $\Cay(T,S)$ of $T$ with
respect to $S$ is  the graph with vertex set $T$ and edge set
$\{\{g,sg\} \,|\,g\in T,s\in S\}$. In particular, $\Cay(T,S)$ is
connected if and only if $T=\langle S\rangle$.
The group $R(T) = \{
\sigma_t|t\in T\}$ of right multiplications $\sigma_t : x \mapsto xt$
is a subgroup of the automorphism group of $\Cay(T,S)$  and acts
regularly on the vertex set. Indeed, a graph is a Cayley graph if and only if it admits a regular group of automorphisms. We note that circulants are precisely the Cayley graphs for cyclic groups.

A graph $\Gamma$ is said to be a \emph{bi-Cayley graph} over a group $H$ if it admits $H$ as a semiregular automorphism
group with two orbits of equal size. (Some authors have used the term \emph{semi-Cayley} instead, see \cite{RJ-1992,LM-1993}.) Moreover, bicirculants are exactly the bi-Cayley graphs over cyclic groups.
The family of  bi-Cayley graphs has been extensively studied, for example cubic bi-Cayley graphs over abelian groups were investigated by Zhou and Feng \cite{ZF-bicay-2014} while the automorphism groups of bi-Cayley graphs were studied in \cite{ZF-bicir-2016}.

 For a graph $\Gamma$, its
\emph{complement} $\overline{\Gamma}$ is the graph with vertex set $V(\Gamma)$,
and two vertices are adjacent if and only if they are not adjacent in $\Gamma$.
We denote the complete graph on $n$ vertices by $\K_n$.

Let $\Gamma_1$ and $\Gamma_2$  be two graphs.
The \emph{lexicographic product} $\Gamma_1[\Gamma_2]$, of $\Gamma_1$ and $\Gamma_2$, is the
graph with vertex set $V(\Gamma_1)\times V(\Gamma_2)$ such that $(u_1,u_2)$ is
adjacent to $(v_1,v_2)$ if and only if either $\{u_1,v_1\}$ is an edge of $\Gamma_1$, or $u_1=v_1$ and
$\{u_2,v_2\}$ is an edge of $\Gamma_2$.

The following theorem is the characterisation of arc-transitive circulants by Kov\'acs \cite{Kovacs-2004} and Li \cite{LCH-circulant-2005}. A \emph{normal circulant} is a Cayley graph $\Cay(T,S)$ such that $T$ is a cyclic group and $R(T)$ is a normal subgroup of $\Aut(\Gamma)$.

\begin{theo}{\rm (\cite[Theorem 1.3]{LCH-circulant-2005})}\label{arccirculant-1}
Let $\Gamma $ be a connected arc-transitive graph  of order $n$ which is not a complete graph and with a cyclic regular group.
Then either
\begin{enumerate}[{\rm (1)}]
\item $\Gamma $ is a normal circulant, or
\item there exists an arc-transitive circulant $\Sigma$  of order $m$, such that $ mb=n$ with $m,b\geqslant 2$, and
 \[\Gamma=\left\{
\begin{array}{ll}
\Sigma[\overline{\K_b}], \ \ or \\
\Sigma[\overline{\K_b}]-b\Sigma, (m,b)=1.
\end{array}\right.\]
\end{enumerate}
\end{theo}

\subsection{The graphs appearing in Theorem \ref{bicirculant-reduction-th2}}

For $m,n\geqslant 2$, we denote   the \emph{complete multipartite graph}  with $m$ parts of size $n$ by   $\K_{m[n]}$, that is, the graph such that
 two vertices are adjacent if and only if they are in distinct parts.  The graph $\K_{2[n]}$ is often denoted by $\K_{n,n}.$

The    \emph{Hamming graph}  $\Ha(d,r)$  has
vertex set $\Delta^d=\{(x_1,x_2,\ldots,x_d)|x_i\in \Delta\}$, where
$\Delta=\{0,1,\ldots,r-1\}$, and  two vertices $v$ and $v'$ are adjacent if and only if they are different
in exactly one coordinate. The    Hamming graph  $\Ha(d,2)$ is  called a \emph{$d$-cube}, and a \emph{folded $d$-cube} is the quotient of $\Ha(d,2)$ by the partition into blocks of size two given by pairs of vertices at distance $d$.
 The complement of the folded $5$-cube is the \emph{Clebsch graph}.

 Let $p$ be an odd prime  and let $e$ be an even integer such that  $e$ divides $p-1$.
Let $\mathbb{Z}_p$ be the group of integers modulo $p$  and let $\sigma$ be a generator of $\Aut(\mathbb{Z}_p)$ $\cong \mathbb{Z}_{p-1}$.
Suppose that $p-1=er$ for some positive integer $r$.
Let $\tau=\sigma^r$ and $S=\{1,1^{\tau},1^{\tau^2},\ldots,1^{\tau^{e-1}}\}\subseteq \mathbb{Z}_p$.
Then $\Cay(p,e)$ denotes the graph $\Cay(\mathbb{Z}_p,S)$.

Let $p$ be an odd prime and let $r$ be a positive   integer  dividing $p-1$. Let $A$ and $A'$ denote two disjoint copies of $\mathbb{Z}_p$ and denote the corresponding elements of $A$ and $A'$ by $i$ and $i'$, respectively.
Let $L(p,r)$ be the unique subgroup of the multiplicative group of $\mathbb{Z}_p$ of order $r$.   We define
two graphs,  $G(2p,r)$ and $G(2,p,r)$, with vertex set $A\cup A'$.  The graph $G(2p,r)$ has edge set $\{\{x,y'\}|x,y\in \mathbb{Z}_p,y-x\in L(p,r)\},$ while  the graph $G(2,p,r)$ (defined only for $r$ even) has  edge set $\{\{x,y\},\{x',y\},\{x,y'\},\{x',y'\}|x,y\in \mathbb{Z}_p,y-x\in L(p,r)\}.$
Note that $G(2,p,r)$ is a non-bipartite graph as it contains a $p$-cycle and is a  2-cover of $\Cay(p,r)$, while  $G(2p,r)$ is bipartite.

For each integer $d\geq 3$ and prime power $q$, let $B(\PG(d-1,q))$ be the bipartite graph with vertices the 1-dimensional and $(d-1)$-dimensional subspaces of a $d$-dimensional vector space over $\GF(q)$, and two subspaces are adjacent if and only if one is contained in the other. We denote the bipartite complement of $B(\PG(d-1,q))$ by $B'(\PG(d-1,q))$, that is, the bipartite graph with the same vertex set but a 1-subspace and a $(d-1)$-subspace  are adjacent if and only if their intersection is the zero subspace.

We define $B'(H(11))$ to be the bipartite graph with vertices the elements of $\mathbb{Z}_{11}$ and the sets $R+i$, where $i\in\mathbb{Z}_{11}$ and
$R=\{1,3,4,5,9\}$, that is, the set of non-zero quadratic residues modulo 11,  such that $n\in\mathbb{Z}_{11}$  is adjacent to $R+i$ if and only if $n\notin R+i$. We note that the bipartite complement of $B'(H(11))$ is isomorphic to $G(22,5)$, see also \cite[p.200]{CO-1987}.

\subsection{Basic lemmas}

We state some  lemmas  which will be used in the
following discussion. The first lemma gives  three  easy observations.

\begin{lemma}\label{circ-bic-1}
\begin{enumerate}[{\rm (1)}]
\item Every circulant of even order is a bicirculant.
\item  A graph is a circulant if and only if its complement is a circulant.
\item  A graph is a bicirculant if and only if its complement is a bicirculant.
\end{enumerate}
\end{lemma}

\begin{lemma}\label{basiccirc-g2p-1}
Let $\Gamma=G(2p,r)$ with $r>1$ and $r$ divides $p-1$. Then the following hold.
\begin{enumerate}[{\rm (1)}]
\item $\Gamma$ is  a Cayley graph of a dihedral group and hence is a bicirculant.
\item   $\Gamma$ is a circulant if and only if
$r$ is even.
\end{enumerate}

\end{lemma}
\begin{proof} Recall that $V(\Gamma)$ consists of the elements $i$ and $i'$ for $i\in\mathbb{Z}_p$. Let
\begin{align*}
\tau: & \  V(\Gamma)\mapsto V(\Gamma),i\mapsto i+1,i'\mapsto (i+1)',\\
\sigma: & \  V(\Gamma)\mapsto V(\Gamma),i\mapsto (-i)', i'\mapsto -i.
\end{align*}

Then $\tau$ is an automorphism of $\Gamma$ of order  $p$ consisting of two $p$-cycles,
and $\sigma$ is an automorphism of $\Gamma$ of order  $2$ swapping the two orbits of $\tau$.
Moreover,  $\sigma \tau \sigma =\tau^{-1}$, and $\langle \sigma,\rho\rangle \cong D_{2p}$ is a dihedral group of order $2p$ that acts regularly on the vertex set.
Thus $\Gamma$ is a Cayley graph of $D_{2p}$, and
 so $\Gamma$ is a bicirculant, and (1) holds.

If $r$ is even, then $i\in L(p,r)$ if and only if $-i\in L(p,r)$.
Hence
\begin{align*}
\rho: & \  V(\Gamma)\mapsto V(\Gamma), i\mapsto i',i'\mapsto i
\end{align*}
 is a  graph automorphism of $\Gamma$ of order 2 and  $\rho \tau=\tau \rho$. Moreover,  $\langle \tau,\rho \rangle\cong \mathbb{Z}_{2p}$ is regular on $V(\Gamma)$. Thus $\Gamma$ is a circulant.

 Now let $r$ be an odd divisor of $p-1$ and suppose that $\Gamma$ is a circulant of the cyclic group  $T$. Then  $\Gamma=\Cay(T,S)$ where $S$ is a subset of $T\setminus \{1_T\}$. Hence $|S|=r$ is an odd integer.
Since $\Gamma$ is undirected, it follows that $S=S^{-1}$, and so $S$ contains the unique involution of $T$.
Suppose that $T\lhd \Aut(\Gamma)$ and let $u=1_T$. Then by \cite[Lemma 2.1]{Godsil-1981}, $\Aut(\Gamma)_u\leqslant \Aut(T)$. Since $\Gamma$ is arc-transitive and $\Gamma(u)=S$, it follows that all elements in $S$ are involutions, a contradiction.
Thus by Theorem \ref{arccirculant-1},
there exists an arc-transitive circulant $\Sigma$  of order $m$, such that $ mb=2p$ with $m,b\in \{2,p\}$, and
 \[\Gamma=\left\{
\begin{array}{ll}
\Sigma[\overline{\K_b}], \ \ or \\
\Sigma[\overline{\K_b}]-b\Sigma.
\end{array}\right.\]

Assume first that $\Gamma=\Sigma[\overline{\K_b}]$. If $(m,b)=(2,p)$, then $\Gamma\cong \K_{p,p}$ has  valency $p>r$, a contradiction.
If $(m,b)=(p,2)$, then each block has 2 vertices, and any two adjacent blocks induce a subgraph  $\K_{2,2}$, and so  $\Gamma$ has even valency, again a contradiction.
Now suppose that $\Gamma=\Sigma[\overline{\K_b}]-b\Sigma$. If $(m,b)=(2,p)$, then $\Gamma\cong \K_{p,p}-p\K_2$ has valency $p-1>r$, a contradiction.
If $(m,b)=(p,2)$, then each block has 2 vertices, and any two adjacent blocks induce a subgraph $2\K_{2}$. Thus $\Gamma$ is a cover of $\Sigma$, and
$\Sigma$ has  valency $r$ and order $p$.
However, both $p$ and $r$ are odd integers, which is impossible.
Hence    $\Gamma$ is  not a circulant when $r$ is an odd divisor of $p-1$.
Therefore,  $\Gamma$ is a circulant if and only if
$r$ is even, so that (2) holds.
\end{proof}

\begin{lemma}\label{basiccirc-1}
A  graph in Theorem \ref{bicirculant-reduction-th2}(a)--(g) is a circulant if and only if
it is one of the following graphs:

 \begin{enumerate}[{\rm (1)}]
\item    $\K_{n,n}$, $n\geqslant 2$;
\item   $\K_n$, $ \K_{n[2]}$,  $n\geqslant 3$;
\item    $\K_{n,n}-n\K_2$, $n\geqslant 3$ is an odd integer;
\item   $G(2p,r)$, where $p$ is a prime and  even $r>1$ divides $p-1$;
\item  $\Cay(p,e)$,  where  $p$ is a prime and $e$ is an even integer dividing  $p-1$.
\end{enumerate}

\end{lemma}
\begin{proof}  Clearly $\K_n$, $\K_{n,n}$ and $\K_{n[2]}$ are circulants for all $n\geqslant 2$, while $\Cay(p,e)$ is a circulant by definition.
By Lemma \ref{basiccirc-g2p-1},  $G(2p,r)$ is a circulant if and only if  $r$ is an even divisor of $p-1$.

It can be easily checked, for example by Magma \cite{Magma-1997}, that $\Ha(2,4)$, $B'(H(11))$, the Clebsch graph and the Petersen graph are not circulants, and by Lemma \ref{circ-bic-1} (2), neither are their complements.

Let $\Gamma=\K_{n,n}-n\K_2$. Then $\Aut(\Gamma)\cong S_n\times S_2$  and when $n$ is odd, this contains a cyclic subgroup of order $2n$ that is regular on the vertex set, and so
$\Gamma$ is a circulant. Moreover, when $n\geqslant 4$ is even, $\Gamma$ is not a circulant, see for example \cite[Theorem 1.1]{ACMX-1996}.

It remains to consider the graphs $B(\PG(d-1,q))$ and $B'(\PG(d-1,q))$, which have automorphism group $\Aut(\PSL(d,q))$.  Suppose that these graphs are circulants. Since $\PGL(d,q)$ acts primitively on each bipartite half and these graphs are neither complete bipartite nor of the form $\K_{n,n}-n\K_2$, the results of Kov\'acs \cite{Kovacs-2004} and Li \cite{LCH-circulant-2005} imply that $\Aut(\Gamma)$ contains a normal cyclic regular subgroup, a contradiction.
 \end{proof}

\begin{lemma}\label{basiccirc-3}
The  graphs in Theorem \ref{bicirculant-reduction-th2}(a)--(g) are bicirculants except for  $\Cay(p,e)$ and  $\K_n$ with  $n$ odd.
\end{lemma}
\begin{proof}   Since the graphs $\Cay(p,e)$ and $\K_n$ with $n$ odd, have an odd number of  vertices, they are not  bicirculants.
It remains to prove that the remaining graphs in Theorem \ref{bicirculant-reduction-th2}(a)--(g) are bicirculants.

By Lemma \ref{basiccirc-1}, $\K_n$ with $n$ even, $\K_{n,n}$ with $n\geqslant 2$ and  $\K_{n[2]}$ with $n\geqslant 3$  are  circulants.
Moreover, all of them have an even number of vertices, and so  by Lemma \ref{circ-bic-1}(1),  they are bicirculants.
Clearly, the Petersen graph and its complement, and $\K_{n,n}-n\K_2$ with $n\geqslant 3$    are bicirculants.

The group $\PGL(d,q)$ has a cyclic subgroup of order  $\frac{q^d-1}{q-1}$ that acts regularly on the set of $1$-dimensional subspaces and the set of hyperplane of $\GF(q)^d$. Thus both $B(\PG(d-1,q))$ and $B'(\PG(d-1,q))$  are bicirculants.

By Lemma \ref{basiccirc-g2p-1},  $G(2p,r)$ is a bicirculant.
Note that $B(H(11))\cong G(22,5)$ is a bicirculant on $22$ vertices. Hence
its complement, $B'(H(11))$, is also a bicirculant by Lemma \ref{circ-bic-1}(3).

If $\Gamma=\Ha(2,4)$, then $\Aut(\Gamma)\cong S_4 \wr S_2$ has a cyclic subgroup of order 8
that is semiregular with two orbits of size $8$ on the vertex set.
Hence $\Ha(2,4)$ is  a bicirculant. Similarly, if $\Gamma$ is the Clebsch graph then $\Aut(\Gamma)\cong \mathbb{Z}_2^4.S_5$ has a cyclic subgroup of order 8
that is semiregular with two orbits of size $8$ on the vertex set.
Thus the Clebsch graph is  a bicirculant.
 \end{proof}

Let $\Gamma$  be a graph with vertex set $V(\Gamma)$ and arc set $A(\Gamma)$.  We define a bipartite graph from $\Gamma$ in the following way.
Let $\widehat{\Gamma}$ be the graph with   vertex set $V(\Gamma)\times \{1,2\}$, and two vertices $(x,1)$ and $(y,2)$ are adjacent if and only if $(x,y)\in A(\Gamma)$. Then the new graph $\widehat{\Gamma}$ is called the  \emph{standard double cover} of $\Gamma$, and it is bipartite with bipartite halves $V(\Gamma)\times \{i\}$ for each $i=1,2$. Note that $\widehat{\Gamma}$ is connected if and only if $\Gamma$ is not bipartite, see \cite[Lemma 3.3]{GLP-1}.

The following lemma shows that all the circulants arising in  Theorem \ref{bicirculant-reduction-th2} are the quotient graphs of some bicirculants.

\begin{lemma}\label{circ-sdc}
If $\Gamma$ is a circulant then the standard double cover of $\Gamma$ is a bicirculant. Moreover, if $\Gamma$ is $G$-arc-transitive then there exists $X\leqslant\Aut(\widehat{\Gamma})$ such that $\widehat{\Gamma}$ is $X$-arc-transitive and $N\lhd X$ such that $\widehat{\Gamma}_N\cong\Gamma$.
\end{lemma}
\begin{proof} Suppose that $g\in\Aut(\Gamma)$. Then $g$ induces an automorphism $\hat{g}$ of $\widehat{\Gamma}$ by $(x,i)^{\hat{g}}=(x^g,i)$. Moreover, if $g$ is an $n$-cycle on $V(\Gamma)$ then $\hat{g}$ has two $n$-cycles on $V(\widehat{\Gamma})$. Thus the first part of the lemma follows.

Note that $\tau:(x,i)\mapsto(x,3-i)$ is an automorphism of $\widehat{\Gamma}$ and $X:=\widehat{G}\times\langle\tau\rangle\leqslant\Aut(\widehat{\Gamma})$ where $\widehat{G}=\{\hat{g}\mid g\in G\}\cong G$. If $\Gamma$ is $G$-arc-transitive then $\widehat{\Gamma}$ is $X$-arc-transitive. Letting $N=\langle \tau\rangle\lhd X$ we see that the orbits of $N$ are $\{(v,i)\mid i\in\{1,2\}\}$ for each $v\in V(\Gamma)$. Moreover, $\widehat{\Gamma}_N\cong\Gamma$.
\end{proof}

All arc-transitive graphs on $2p$ vertices for a prime $p$ are given by the following result.

\begin{theo}\label{bicirc-2p-1}\cite[Theorem 2.4]{CO-1987}
Let $\Gamma$ be a connected arc-transitive  graph.
If $|V(\Gamma)|=2p$ for some prime number $p$, then   $\Gamma$ is one of the following graphs:

\begin{enumerate}[{\rm (1)}]
\item $\K_{2p}$ or $\K_{p,p}$;

\item the Petersen graph or its complement;

\item $G(2,p,r)$ for some even integer $r$ dividing $p-1$;

\item $G(2p,r)$ for some  integer $r>1$ dividing $p-1$;

\item $B(\PG(n-1,q))$ or  $B'(\PG(n-1,q))$, and $p=\frac{q^n-1}{q-1}$;

\item $B'(H(11))$ and $p=11$.
\end{enumerate}

\end{theo}

We also give the following well-known lemma.
Note that every connected bipartite graph has a unique bipartite partition.
 For a $G$-vertex-transitive bipartite graph, we use $G^+$
to denote the stabiliser in $G$ of the two bipartite halves.

\begin{lemma}\label{lem:Knn}
Let $\Gamma$ be a $G$-arc-transitive connected graph such that $G$ is bi-quasiprimitive on $V(\Gamma)$. If $G^+$ acts unfaithfully on each orbit, then $\Gamma\cong \K_{n,n}$ for some $n\geqslant 2$.
\end{lemma}
\begin{proof}
Let $\Delta_0$ and $\Delta_1$ be the bipartite halves of $\Gamma$. Let
$G_{(\Delta_0)}^{+}$ be the kernel of $G^+$ on $\Delta_0$ and $G_{(\Delta_1)}^{+}$ be the kernel of $G^+$ on $\Delta_1$.
Suppose that $G^+$ acts unfaithfully on each $\Delta_i$, that is, $G_{(\Delta_i)}^{+}\neq 1$ for $i=0,1$.   Also, let $\sigma\in G$ be an element interchanging $\Delta_0$ and $\Delta_1$ and note that $G=\langle G^+,\sigma\rangle$. Then $(G_{(\Delta_i)}^{+})^\sigma=G_{(\Delta_{1-i})}^{+}$  and so
$1 \neq G_{(\Delta_0)}^{+}\times G_{(\Delta_1)}^{+}\vartriangleleft G$.
Since $G$ is bi-quasiprimitive on $V(\Gamma)$ and $G_{(\Delta_0)}^{+}\times G_{(\Delta_1)}^{+}$ is not transitive, it follows that
$G_{(\Delta_0)}^{+}\times G_{(\Delta_1)}^{+}$ has two orbits on $V(\Gamma)$, namely $\Delta_0$ and $\Delta_1$. Thus $G_{(\Delta_0)}^{+}$ fixes each vertex in $\Delta_0$ and is transitive on $\Delta_1$. Since $G$ is arc-transitive it follows that $\Gamma=\K_{n,n}$ where $n=|\Delta_i|\geqslant 2$.
\end{proof}

\subsection{Some group theory}

The classification of    primitive permutation groups  that contain a
cyclic regular subgroup was independently obtained by Jones \cite{Jones-2002} and Li \cite[Corollary 1.2]{LCH-abelianregular-2003}.
Moreover, by \cite[Theorem 1.2]{LP-circulant-2012}, every   quasiprimitive group with a regular cyclic  subgroup is  primitive.
Hence the family of  quasiprimitive groups with a regular cyclic  subgroup is also  completely  determined in the following lemma.

\begin{lemma}{\rm (\cite{Jones-2002},\cite[Corollary 1.2]{LCH-abelianregular-2003})}\label{primitive-cyclic-1}
A primitive permutation group $G$ of degree $n$  contains a
cyclic regular subgroup if and only if one of the  following holds.
 \begin{enumerate}[{\rm (i)}]
\item $\mathbb{Z}_p\leqslant G\leqslant \AGL(1,p)$, where $n=p$ is a prime;
\item $G=A_n$ with $n\geqslant 5$ odd, or $S_n$, where $n\geqslant 4$;
\item $\PGL(d,q)\leqslant G\leqslant \PGammaL(d,q)$ and $n=(q^d-1)/(q-1)$;
\item $(G,n)=(\PSL(2,11),11)$, $(M_{11},11)$, $(M_{23},23)$.
\end{enumerate}

Moreover, in cases {\rm (ii)}--{\rm (iv)} $G$ is $2$-transitive.
\end{lemma}

The following result about primitive permutation groups  that contain an
element with exactly two equal cycles is due to M\"{u}ller.

\begin{theo}{\rm (\cite[Theorem 3.3]{Muller-2018})}\label{bicirculant-primitive-2}
Let $G$ be a primitive permutation group of degree $2n$ that contains an
element with exactly two cycles of length $n$. Then one of the following holds, where $G_0$ denotes the stabiliser of a point.
\begin{enumerate}[$(1)$]
\item  (Affine action) $\mathbb{Z}_2^m\lhd G\leqslant \AGL(m,2)$ is an affine permutation group, where $n=2^{m-1}$. Further, one of the following holds.
 \begin{enumerate}[{\rm (a)}]
\item $n= 2$, and $ G_0= \GL(2,2)$;
\item $n=2$, and $ G_0= \GL(1,4)$;
\item $n=4$, and $ G_0= \GL(3,2)$;
\item $n=8$, and $G_0\in \{\mathbb{Z}_5:\mathbb{Z}_4,\GammaL(1,16),(\mathbb{Z}_3\times \mathbb{Z}_3):\mathbb{Z}_4,\SigmaL(2,4),\GammaL(2,4),A_6,\GL(4,2),\\(S_3\times S_3):\mathbb{Z}_2,S_5,S_6,A_7\}.$\footnote{We note that \cite[Theorem 3.3]{Muller-2018} gives $|\GammaL(1,16):G_0|=3$ as one of the possibilities but there is a unique such group, namely the $\mathbb{Z}_5:\mathbb{Z}_4$ that we list here.}
\end{enumerate}
\item (Almost simple action) $S\leqslant G\leqslant \Aut(S)$ for a  nonabelian simple group $S$, and one of the following holds.
 \begin{enumerate}[{\rm (a)}]
\item  $n\geqslant 3$ and  $A_{2n}\leqslant G\leqslant S_{2n}$ in its natural action;
\item $n=5$ and $A_5\leqslant G\leqslant S_5$ in the action on the set of  $2$-subsets of $\{1,2,3,4,5\}$;
\item $n=(q^d-1)/2(q-1)$  and $\PGL(d,q)\leqslant G\leqslant  \PGammaL(d,q)$ for some odd  prime power $q$ and $d\geqslant 2$ even;
\item $n=11$ and $M_{22}\leqslant G\leqslant \Aut(M_{22})$;
\item $n=6$ and $G=M_{12}$;
\item $n=12$ and $G=M_{24}$.
\end{enumerate}
\end{enumerate}
\end{theo}

\begin{lemma}\label{lem:2trans}
Let $\Gamma$ be a $G$-arc-transitive bipartite graph of order $2n$ with $n$ not a prime, such that $G^+$ acts faithfully and $2$-transitively on each bipartite half.
\begin{enumerate}[{\rm (1)}]
\item If $G^+$ contains a cyclic subgroup that is transitive on each bipartite half then $\Gamma=\K_{n,n}-n\K_2$ for some $n\geqslant 3$, $B(\PG(d-1,q))$ or $B'(\PG(d-1,q))$ for some $d\geq 3$ and $q$ a prime power.
\item If $G^+$ is almost simple and contains a cyclic subgroup with two equal sized orbits on each bipartite half then $\Gamma=\K_{12,12}$, $\K_{n,n}-n\K_2$ for some $n\geqslant 3$,  $B(\PG(d-1,q))$ or $B'(\PG(d-1,q))$ for some even $d\geq 3$ and  odd prime power $q$.
\end{enumerate}
\end{lemma}
\begin{proof}
Let $\Delta_0$  and $\Delta_1$ be the bipartite halves of $\Gamma$ and let $u\in \Delta_0$. If all 2-transitive actions of $G^+$ on $n$ points are equivalent then there exists $v\in \Delta_1$ such that $(G^+)_u=(G^+)_v$. Moreover, $(G^+)_u$ is transitive on both  $\Delta_1\backslash\{v\}$ and $\Delta_0\backslash\{u\}$. Hence $\Gamma=\K_{n,n}-n\K_2$ for some $n\geqslant 3$.  It remains to consider the case where the action of $G^+$ on $\Delta_0$ and $\Delta_1$ are inequivalent 2-transitive actions on $n$ points.

Suppose first that $G^+$ contains a cyclic subgroup  that is transitive on $\Delta_0$. Then recalling that $n$ is not a prime and $G^+$ is primitive on $\Delta_0$, we have from Lemma \ref{primitive-cyclic-1}, that one of the following holds:
\begin{enumerate}
\item $G^+= S_6$ and $n=6$;
\item $\PGL(d,q)\lhd G^+\leqslant \PGammaL(d,q)$ and $n=(q^d-1)/(q-1)$ with $d\geq 3$.
\end{enumerate}
 In the first case, an element acting on $\Delta_1$ as a 6-cycle acts on $\Delta_2$ as a product of a 2-cycle and a 3-cycle, and so no suitable cyclic subgroup exists.
 In the second case, $\PGL(d,q)\lhd G^+$ and $\Delta_0$ and $\Delta_1$ correspond to the set of $1$-spaces and $(d-1)$-spaces of a $d$-dimensional vector space over $\GF(q)$. Thus $\Gamma=B(\PG(d-1,q))$ or $B'(\PG(d-1,q))$.

Next suppose that $G^+$ is almost simple and contains a cyclic subgroup with two equal sized orbits  on $\Delta_0$. Then Theorem \ref{bicirculant-primitive-2} implies that one of the following holds:
\begin{enumerate}
\item $A_6\lhd G^+\leqslant S_6$ and $n=6$;
\item $\PGL(d,q)\lhd G^+\leqslant \PGammaL(d,q)$ and $n=(q^d-1)/(q-1)$ with even $d\geq 3$  and odd $q$;
\item $G^+=M_{12}$ and $n=12$.
\end{enumerate}
In the first case, if $g\in G^+$ acts on $\Delta_0$ as two disjoint 3-cycles then it has three fixed points on $\Delta_2$, a contradiction. We have also already determined the graph in the second case.
In the last case, $(G^+)_u$ is transitive on $\Delta_1$ and so $\Gamma=\K_{12,12}$.
\end{proof}

\section{Reduction result and circulants}

In this section, we  give a reduction result for the family of
arc-transitive bicirculants.
Let $\Gamma$ be a $G$-vertex-transitive  bicirculant over a cyclic subgroup $H$ of $G$. Then the first lemma presents a
relationship between the two orbits of  $H$  and the blocks for the action of $G$ on the vertex set.

\begin{lemma}\label{quo-1}
Let $\Gamma$ be a $G$-vertex-transitive   bi-Cayley graph over the  subgroup $H$ of $G$.
Let $H_0$ and $H_1$ be the two orbits of $H$ on $V(\Gamma)$ and
let $\mathcal{B}$ be a $G$-invariant partition of $V(\Gamma)$. Then the following hold.

\begin{enumerate}[$(1)$]
\item Either all elements of $\mathcal{B}$ are subsets of $H_0$ or $H_1$;
or  $B\cap H_0\neq \varnothing$ and $B\cap H_1\neq \varnothing$ for every   $B\in \mathcal{B}$.

\item If $B\in \mathcal{B}$ and  $B\cap H_i\neq \varnothing$ for some $i$, then $B\cap H_i$ is a block for $H$ on $H_i$.
\end{enumerate}

\end{lemma}
\begin{proof}  (1) Suppose that there exists some  $B\in \mathcal{B}$ such that  $B\subseteq H_i$ for some  $i\in \{0,1\}$, and assume  that there is another block $B'\in \mathcal{B}$ such that  $B'\cap H_0\neq \varnothing$ and $B'\cap H_{1}\neq \varnothing$.
Then for each vertex $b\in B$, there exists $h\in H$ such that  $b^{h}\in B'\cap H_i$, as  $H$ acts transitively on  $H_i$. Thus $b^{h}\in B'\cap B^{h}$. Since $B^{h}\in \mathcal{B}$
and $\mathcal{B}$ is a block system,  we get $B'=B^{h}\subseteq H_i$, a contradiction. Therefore either all elements of $\mathcal{B}$ are subsets of  $H_0$ or $H_1$, or the intersection of each $B\in \mathcal{B}$ with $H_0$ and $H_1$ is non-empty.

(2) Let $B\in \mathcal{B}$ and suppose that $B\cap H_i\neq \varnothing$ for some $i$. Let $h\in H$. Then $B^{h}= B$ or $B^{h}\cap  B=\varnothing$. Since $(B\cap H_i)^{h}\subseteq H_i$, it follows that
$(B\cap H_i)^{h}=B\cap H_i$ or $(B\cap H_i)^{h}\cap (B\cap H_i)=\varnothing$ and so $B\cap H_i$ is a block for $H$ on $H_i$.
 \end{proof}

\begin{lemma}\label{bicirc-qu-1}
Let $\Gamma$ be a  $G$-vertex-transitive  bi-Cayley graph   with at least three vertices  over  the abelian  subgroup $H$ of $G$, and let
$H_0$ and $H_1$ be the two orbits of $H$ on $V(\Gamma)$.
Let $N$ be a  normal subgroup of $G$ maximal with respect to having at least $3$ orbits. Then $\Gamma_N$ is a $G/N$-vertex-transitive graph, $G/N$ is faithful on $V(\Gamma_N)$
and the following hold.

\begin{enumerate}[$(1)$]
\item If   each $N$-orbit is a subset of  either $H_0$ or $H_1$, then  $\Gamma_N$ is a  bi-Cayley graph over $HN/N$.
 In particular, if $H$ is a cyclic group, then $\Gamma_N$  is a bicirculant of $HN/N$.

\item If  each $N$-orbit intersects both $H_0$ and $H_1$ non-trivially, then  $\Gamma_N$  is isomorphic to a Cayley graph  of $HN/N$.
 In particular, if $H$ is a cyclic group, then $\Gamma_N$  is isomorphic to a circulant of $HN/N$.
\end{enumerate}

\end{lemma}
\begin{proof}
Let $K$ be the kernel of the action of $G$ on $V(\Gamma_N)$. Then $N\leqslant K$. If $N<K$, then
$K$ has 1 or 2 orbits on $V(\Gamma)$, which is impossible. Thus  $N=K$. Since $\Gamma$ is  $G$-vertex-transitive, $\Gamma_N$ is  $G/N$-vertex-transitive.
Let $\mathcal{B}=\{B_1,\ldots,B_t\}$ be the set of $N$-orbits.

(1) Suppose that each $N$-orbit is a subset of  either $H_0$ or $H_1$.
Then $|B_i|$ divides $|H|$, and without loss of generality, we may assume that   $H_0=B_1\cup \cdots\cup B_r$ and
$H_1=B_{r+1}\cup \cdots\cup B_{2r}$ where $2r=t$.
Since $H$ is  an abelian group,
$HN/N\cong H/(H\cap N)$ is an abelian group. Since $H$ is transitive on both $H_0$ and $H_1$,  it follows that  $HN/N$ is regular  on both sets $\{B_1,\ldots,B_r\}$ and $\{B_{r+1},\ldots,B_{2r}\}$.
 Therefore,
$\Gamma_N$ is a bi-Cayley graph  of $HN/N$.

(2) Suppose that each $N$-orbit intersects both $H_0$ and $H_1$ non-trivially.
Then  by   Lemma \ref{quo-1}(2),  for each $B\in\mathcal{B}$ we have that $B\cap H_i$ is a block  for $H$ on $H_i$ for each  $i\in \{0,1\}$.
As $H$ is transitive on  $H_0$,   $HN/N$ is transitive on $\{B_1\cap H_0,\ldots,B_t\cap H_0\}$.   Since  $N$ is the kernel of the  action of $G$ on $V(\Gamma_N)$,
$H\cap N$ is the kernel of the  action of $H$ on $\{B_1\cap H_0,\ldots,B_t\cap H_0\}$.
Since $H$ is an abelian group,
it follows that $HN/N$ acts  regularly  on $\{B_1\cap H_0,\ldots,B_t\cap H_0\}$. Thus $HN/N$ acts regularly  on $\mathcal{B}$, and so
$\Gamma_N$ is  a Cayley graph  of $HN/N$.
In particular, if $H$ is a cyclic group, then
 $\Gamma_N$ is    a   circulant of $HN/N$.  \end{proof}

We are ready to give a  reduction result.

\begin{theo}\label{bic-reduction-th2}
Let $\Gamma$ be a connected $G$-arc-transitive  bicirculant  with at least three vertices over   the    cyclic subgroup $H$ of $G$.
Let $N$ be a  normal subgroup of $G$ maximal with respect to having at least $3$ orbits. Then
$\Gamma$ is an $r$-cover of $\Gamma_N$ where $r$ divides the valency of $\Gamma$,
$\Gamma_N$ is $G/N$-arc-transitive and  is either a circulant or a bicirculant over $HN/N$, and $G/N$ is faithful and either quasiprimitive or bi-quasiprimitive on $V(\Gamma_N)$.

\end{theo}
\begin{proof} Since $N$ is a  normal subgroup of $G$ maximal with respect to having at least $3$ orbits, it follows from  Lemma \ref{bicirc-qu-1} that $G/N$ is faithful  on $V(\Gamma_N)$. Moreover, all normal subgroups of $G/N$ are transitive or have two orbits on $V(\Gamma_N)$. Thus $G/N$ is quasiprimitive or bi-quasiprimitive on $V(\Gamma_N)$.
Let $\mathcal{B}=\{B_1,\ldots,B_t\}$ be the set of $N$-orbits, so that $t\geqslant 3$. Since $\Gamma$ is $G$-arc-transitive, it follows that $\Gamma_N$ is a $G/N$-arc-transitive graph,
and  each induced subgraph $[B_i]$ is  empty.  For   an arc $(B_1,B_2)$ of $\Gamma_N$, there exists     a vertex $b_1$ of $\Gamma$ such that  $b_1\in B_1$ and
$\Gamma(b_1)\cap B_2\neq \varnothing$. For any $b_1'\in B_1$, we have $b_1'=b_1^{n}$ for some $n\in N$. Hence
$$\Gamma(b_1')\cap B_2=\Gamma(b_1)^n\cap B_2=(\Gamma(b_1)\cap B_2)^n\neq \varnothing.$$
Thus  the number of  vertices in $B_2$ adjacent to a given  vertex of $B_1$ is constant.
If  $\Gamma(b_1)\subseteq B_2$, then by the connectivity of $\Gamma$, we have $t=2$, a contradiction.
Hence  $\Gamma(b_1)\nsubseteq B_2$  and so there exists  $B_3\in \mathcal{B}\backslash\{B_1,B_2\}$ such that $B_1$ and $B_3$ are adjacent in $\Gamma_N$.
Let $b_2\in B_2$ and $b_3\in B_3$ such that $(b_1,b_2)$ and $(b_1,b_3)$ are  arcs of $\Gamma$. Since   $\Gamma$ is $G$-arc-transitive, there exists $g\in G_{b_1}$ such that
$b_2^g=b_3$. Since $B_1,B_2,B_3$ are blocks of $G$, it follows that $B_1^g=B_1$ and $B_2^g=B_3$. Thus $(\Gamma(b_1)\cap B_2)^g=\Gamma(b_1)\cap B_3$, and so
$|\Gamma(b_1)\cap B_2|=|\Gamma(b_1)\cap B_i|$ for any $i$ such that $(B_1,B_i)$ is an arc of $\Gamma_N$. We also showed above that
$|\Gamma(b_1)\cap B_2|=|\Gamma(b_1')\cap B_2|$ for all $b_1,b_1'\in B_1$.
 Therefore,
$\Gamma$ is an $r$-cover of $\Gamma_N$ where $r=|\Gamma(b_1)\cap B_2|$ is a divisor of $|\Gamma(b_1)|$.

Let $H_0$ and $H_1$ be the two orbits of $H$ on $V(\Gamma)$. If  $B_1\subset H_i$ for some $i\in \{0,1\}$, then by Lemma \ref{quo-1}(1),
for each $B'\in \mathcal{B}$, either $B'\subset H_0$ or $B'\subset H_1$.
It follows from Lemma \ref{bicirc-qu-1} that $\Gamma_N$ is a bicirculant over $HN/N\cong H/(H\cap N)$.
Finally,  suppose that
$B\cap H_0\neq \varnothing$ and $B\cap H_1\neq \varnothing$ for some $B\in \mathcal{B}$.
Then again by Lemma \ref{quo-1}(1), $B'\cap H_0\neq \varnothing$ and $B'\cap H_1\neq \varnothing$ for every $B'\in \mathcal{B}$.
Hence, by  Lemma \ref{bicirc-qu-1},  $\Gamma_N$  is isomorphic to a  circulant of $HN/N$.
\end{proof}

Our next  proposition determines the family of arc-transitive circulants that are  vertex quasiprimitive or vertex bi-quasiprimitive.

\begin{prop}\label{circ-biqp-th1}
Let $\Gamma$ be a connected $G$-arc-transitive circulant over a cyclic subgroup $H$ of $G$. Then
 the following statements hold:

\begin{enumerate}[$(1)$]
\item If  $G$ is quasiprimitive on $V(\Gamma)$, then  $\Gamma$ is one of the following two graphs:

 \begin{itemize}
\item[(1.1)]  a complete graph;

\item[(1.2)]  $\Cay(p,e)$ where $e$ is an even integer dividing  $p-1$.

\end{itemize}

\item If $G$ is bi-quasiprimitive on $V(\Gamma)$, then  $\Gamma$ is one of the following three graphs:

 \begin{itemize}
\item[(2.1)]  $\K_{n,n}$;

\item[(2.2)] $ \K_{n,n}-n\K_2$ where $n$ is an odd integer;

\item[(2.3)] $G(2p,r)$ where   $p$ is a prime and $r$ is an even divisor of $p-1$.

\end{itemize}
\end{enumerate}

\end{prop}
\begin{proof}
(1) Suppose that       $G$ is quasiprimitive on $V(\Gamma)$. Then by   \cite[Theorem 1.2]{LP-circulant-2012},
$G$ is primitive on $V(\Gamma)$.  Moreover, $G$ is  listed
in  Lemma \ref{primitive-cyclic-1}: either $|V(\Gamma)|=p$ is a prime and $G\leqslant \AGL(1,p)$, or $G$ is $2$-transitive on $V(\Gamma)$. If $G$ is $2$-transitive on $V(\Gamma)$, then $\Gamma$ is a complete graph, and case (1.1) holds;
if $|V(\Gamma)|=p$ is a prime and $G\leqslant \AGL(1,p)$, then
by  Chao \cite{Chao-1971}, $\Gamma$ is a Cayley graph $\Cay(p,e)$ where $e$ is an even integer dividing  $p-1$, that is, (1.2) occurs.

(2) Suppose that $G$ is bi-quasiprimitive on $V(\Gamma)$. Then $G$ has a minimal normal subgroup  that has exactly two orbits  on $V(\Gamma)$, say $\Delta_0$ and $\Delta_1$. Hence  $|H|$ is even.
Since $\Gamma$ is  $G$-arc-transitive and connected, it follows that each $\Delta_i$ does not contain any edge of $\Gamma$, and so
$\Gamma$ is a bipartite graph with  $\Delta_0$ and $\Delta_1$ being  the two bipartite halves.
Let $G^+=G_{\Delta_0}=G_{\Delta_1}$.
Since $G$ is transitive on the vertex set, it follows that  $G^+$ is a normal subgroup of $G$ of index 2, and so $G=\langle G^+, \sigma \rangle$	for some element $\sigma\in G\setminus G^+$ with $\sigma^2\in G^+$.
Moreover, $\Delta_0^{\sigma}=\Delta_1$ and $\Delta_1^{\sigma}=\Delta_0$.
Since
$H$ is cyclic and regular on $V(\Gamma)$,  the group $H^+=H\cap G^+$ is transitive and so regular on each $\Delta_i$.

Suppose that $G^+$ is not faithful on $\Delta_0$ and $\Delta_1$. Then by Lemma \ref{lem:Knn}, $\Gamma\cong \K_{|H|/2,|H|/2}$, and (2.1) holds. In the remainder,  we assume that
$G^+$ is  faithful on $\Delta_0$ and $\Delta_1$.
Suppose first that  $G^+$ is not quasiprimitive on each $\Delta_i$.
Let $N$ be a maximal intransitive normal subgroup of $G^+$ on
$\Delta_0$. Then $N^\sigma$ is a maximal intransitive normal subgroup of $G^+$ on
$\Delta_1$. Let $\mathcal{B}_0$ be the set of $N$-orbits on
$\Delta_0$, and let $\mathcal{B}_1$ be the set of $N^\sigma$-orbits on
$\Delta_1$. Then $|\mathcal{B}_0|=|\mathcal{B}_1|\geqslant 2$. Let $\mathcal{B}=\mathcal{B}_0\cup \mathcal{B}_1$. Then
as $N^{\sigma^2}=N$,
$\mathcal{B}$ is a $G$-invariant partition of $V(\Gamma)$.
Let $K$ be the kernel of $G$ acting on $\mathcal{B}$. Then $K$ is a normal subgroup of $G$ with at least 4 orbits on $V(\Gamma)$. Since
$G$ is bi-quasiprimitive, we have $K=1$.  Thus $G$ and hence $H$ act faithfully on $\mathcal{B}$. Since $H$ is abelian and transitive on $V(\Gamma)$ it follows that $H$
is transitive and so regular  on  $\mathcal{B}$. Thus $|\mathcal{B}|=|H|=|V(\Gamma)|$, a contradiction.
Hence  $G^+$ is quasiprimitive on each $\Delta_i$.

Since
$H$ is cyclic and transitive on $V(\Gamma)$, $H^+$ is cyclic,  transitive and so regular on each $\Delta_i$. It follows from \cite[Theorem 1.2]{LP-circulant-2012}  that  $G^+$ is primitive on $\Delta_i$.
Then by Lemma \ref{primitive-cyclic-1},  either $|H|=2p$ for some prime $p$ and $G^+\leqslant \AGL(1,p)$, or
$G^+$ is $2$-transitive on $\Delta_i$.

If  $|H|=|V(\Gamma)|=2p$ for some prime $p$, then as $\Gamma$ is  connected $G$-arc-transitive,   $\Gamma$ is one of the graphs listed in  Theorem \ref{bicirc-2p-1}.
By Lemma \ref{basiccirc-1}, $B'(H(11))$, $B(\PG(n-1,q))$ and  $B'(\PG(n-1,q))$ are not circulants, while $\K_{2p}$, $G(2,p,r)$, the Petersen graph and  its complement are not bipartite. Thus $\Gamma$ is $\K_{p,p}$  in which case (2.1) holds, or  $\Gamma$ is $G(2p,r)$, in which case  $r$ is even by Lemma \ref{basiccirc-g2p-1},  that is, (2.3) holds.

Finally, assume that $|H|=|V(\Gamma)|\neq 2p$ for any prime $p$. Then  $G^+$ is 2-transitive on both $\Delta_0$ and $\Delta_1$, hence  $\Gamma$ is given by Lemma \ref{lem:2trans} (1). However, by Lemma \ref{basiccirc-1}, the graphs $B(\PG(d-1,q))$ and $B'(\PG(d-1,q))$ are not circulants, and so
$\Gamma=\K_{n,n}-n\K_2$, in which case $n$ must be odd by Lemma \ref{basiccirc-1}, and  case (2.2) holds.
\end{proof}

\section{Vertex quasiprimitive arc-transitive bicirculants}

This section is devoted   to classifying all vertex quasiprimitive arc-transitive bicirculants.
The first lemma  gives a useful observation on quasiprimitive  imprimitive
permutation groups that contain a cyclic subgroup
with  exactly two orbits.

\begin{lemma}\label{bicirc-quasi-1}
Let $G$ be a quasiprimitive imprimitive 
permutation group on a set $\Omega$. Suppose that $H$ is a cyclic subgroup of $G$
that has  exactly two orbits, $H_0$ and $H_1$, on $\Omega$ and  $|\Omega|=2|H|$.  Let  $\mathcal{B}$ be  a non-trivial maximal   $G$-invariant partition of $\Omega$.
Then $G$ is faithful and primitive on $\mathcal{B}$, and also the following hold:

\begin{enumerate}[{\rm (1)}]
\item $H$ is regular on $\mathcal{B}$, and for each   $B\in \mathcal{B}$ we have $|B|=2$ and  $|B\cap H_0|=1=|B\cap H_1|$;
\item $G$ and $|H|$ are  as in   Lemma \ref{primitive-cyclic-1}, and in particular, either $G$ is a $2$-transitive group on $\mathcal{B}$ or
$|H|=p$ for some prime $p$.
\end{enumerate}

\end{lemma}
\begin{proof}  Since $G$ is quasiprimitive on $\Omega$, it follows that $G$ acts faithfully on $\mathcal{B}$. Moreover, the maximality of $\mathcal{B}$
implies that $G$ is primitive on $\mathcal{B}$.
Suppose that for each $B\in \mathcal{B}$,  either $B\subseteq H_0$ or $B\subseteq H_1$.
Then $|B|$ divides $|H|$ and we may
assume that   $H_0=B_1\cup \cdots\cup B_r$ and
$H_1=B_{r+1}\cup \cdots\cup B_{2r}$.
Since  $H_0$ and $H_1$ are the two orbits of $H$, it follows that  $H$ is transitive on both sets $\{B_1,\ldots,B_r\}$ and $\{B_{r+1},\ldots,B_{2r}\}$.
Since $G$ is faithful on $\mathcal{B}$, $H$ is
faithful  on $\mathcal{B}$.
Further, as $H$ is a cyclic group, it is regular on each orbit on $\mathcal{B}$.
Thus $|\mathcal{B}|=2|H|=|\Omega|$,
contradicting the fact that $\mathcal{B}$ is a non-trivial  $G$-invariant partition of $\Omega$. Hence there exists  $B\in \mathcal{B}$ such that
$B\cap H_0\neq \varnothing$ and $B\cap H_1\neq \varnothing$.
By Lemma \ref{quo-1}, for all $B\in \mathcal{B}$, we have
$$B\cap H_0\neq \varnothing \quad \textrm{and} \quad B\cap H_1\neq \varnothing.$$

Let $\mathcal{B}=\{B_1,\ldots,B_t\}$.
For $i\in \{0,1\}$, let $\mathcal{B}_i=\{B_1\cap H_i,\ldots,B_t\cap H_i\}$. Since each $B_j$ meets each $H_i$ non-trivially,  we have that $|\mathcal{B}_0|=|\mathcal{B}_1|=t$.
Moreover, as $H$ is transitive on each $H_i$, it is transitive on each $\mathcal{B}_i$. Since $H$ is cyclic,
it has  a unique subgroup of each order and so the kernel of $H$ on $\mathcal{B}_0$
is equal to the kernel of $H$ on $\mathcal{B}_1$, and so is  in the kernel of $H$ on $\mathcal{B}$.
It follows that $H$ acts faithfully and hence regularly on each $\mathcal{B}_i$. Thus
$|H|=t=|\mathcal{B}_i|$, and so
each $B\in \mathcal{B}$ has size  $2$. Hence (1)  holds. Furthermore, the action of $G$ on $\mathcal{B}$ satisfies the conditions of
Lemma \ref{primitive-cyclic-1}, so
$G$ and $|H|$ are  as in  Lemma \ref{primitive-cyclic-1}, and
hence either $G$ is 2-transitive on $\mathcal{B}$ or $|\mathcal{B}|=p$ is a prime, so that (2) holds.
 \end{proof}

\medskip
The following  proposition determines all the vertex quasiprimitive arc-transitive  bicirculants.

\begin{prop}\label{bicirculant-quasiprimitive-th1}
Let $\Gamma$ be a connected $G$-arc-transitive  bicirculant over  the cyclic subgroup $H$ of order $n$ such
that $G$ is quasiprimitive on $V(\Gamma)$. Then one of the following holds:
\begin{enumerate}[$(1)$]
\item  $G$ is primitive on $V(\Gamma)$ and $\Gamma$ is one of the following graphs:

 \begin{itemize}
\item[(1.1)]  $\K_{2n}$, and  $G$ is a $2$-transitive group  of degree $2n$ as in Theorem \ref{bicirculant-primitive-2};

\item[(1.2)]  Petersen graph or its complement, and $A_5\leqslant G\leqslant S_5$;

\item[(1.3)] $\Ha(2,4)$ or its complement,  and $G$ is a rank $3$ subgroup of $\AGL(4,2)$;
\item[(1.4)]  Clebsch graph or its complement,  and $G$ is a rank $3$ subgroup of $\AGL(4,2)$.

\end{itemize}

\item $G$ is not primitive on $V(\Gamma)$ and $\Gamma$ is one of the following graphs:

 \begin{itemize}
\item[(2.1)]  $\K_{n[2]}$ and  $G$ has rank $3$ on vertices;

\item[(2.2)] $ \K_{n,n}-n\K_2$ with  $\PGL(d,q)\leqslant G\leqslant \PGammaL(d,q)$ and $n=(q^d-1)/(q-1)$.

\end{itemize}
\end{enumerate}

\end{prop}
\begin{proof}   Suppose first that $G$ is primitive on $V(\Gamma)$.
Then $G$ is given in Theorem \ref{bicirculant-primitive-2}. If $G$ is 2-transitive on $V(\Gamma)$, then $\Gamma$ is isomorphic to a complete graph
$\K_{2n}$, so that (1.1) holds.
Now assume  that  $G$ is not 2-transitive  on $V(\Gamma)$.  Then $n$, $G$  and the vertex stabiliser $G_u$ are one of the following:

 \begin{itemize}
\item[(i)]  $n=8$, $G\leqslant \AGL(4,2)$ and $G_u\in \{\mathbb{Z}_5:\mathbb{Z}_4,(\mathbb{Z}_3\times \mathbb{Z}_3):\mathbb{Z}_4, (S_3\times S_3):\mathbb{Z}_2,S_5\}$.

\item[(ii)] $n=5$, and $A_5\leqslant G\leqslant S_5$ in the action on the set of  2-subsets of $\{1,2,3,4,5\}$.

\end{itemize}

If $\Gamma$ is in case (ii), then by \cite[p.75]{DM-1}, $\Gamma$ is the Petersen graph or its complement, and (1.2) holds.
Now we determine the graphs in case (i) by \textsc{Magma} \cite{Magma-1997}. In all cases $G$ has rank three. If $G_u= (S_3\times S_3):\mathbb{Z}_2$ or $(\mathbb{Z}_3\times \mathbb{Z}_3):\mathbb{Z}_4$, then  $\Gamma=\Ha(2,4)$ or its complement, and
(1.3) holds.
If  $G_u= \mathbb{Z}_5:\mathbb{Z}_4$ or $ S_5$, then $\Gamma$ is  the Clebsch graph or its complement (the folded $5$-cube), so  (1.4) holds.


It remains to consider the case that   $G$ is not primitive on $V(\Gamma)$.
Let  $\mathcal{B}$ be a non-trivial maximal  block system of $G$ on $V(\Gamma)$.
Then by  Lemma \ref{bicirc-quasi-1},    for each  $B\in \mathcal{B}$, we have that
$|B\cap H_0|=1=|B\cap H_1|$  and $|B|=2$. Moreover, $H$ is regular on $\mathcal{B}$,  $G$ is faithful and primitive on $\mathcal{B}$,
and the pair $(G,n)$ is as in Lemma \ref{primitive-cyclic-1}.

If the pair $(G,n)$ is as in case (i) of Lemma \ref{primitive-cyclic-1}, then $|V(\Gamma)|=2p$ for some prime $p$ and $H$ is a normal subgroup of $G$ of order $p$.
However, this contradicts $G$ being  quasiprimitive on a set of size $2p$.
Moreover,  $(G,n)\notin \{(\PSL(2,11),11),(M_{23},23)\}$, as in these cases $G$ does not have a transitive action on $2n$ points.

It remains to assume that   $(G,n)$ is either $(M_{11},11)$ or   is as in   cases (ii)--(iii) of Lemma \ref{primitive-cyclic-1}.
Thus $G$ is $2$-transitive  on $\mathcal{B}$ and $G$ has a quasiprimitive action on $2n$ points.
Hence  the quotient graph $\Gamma_\mathcal{B}$ is a complete graph on  $n$ vertices for some  $n\geqslant 4$.
Let $B_1,B_2\in \mathcal{B}$. Then $B_1$ and $B_2$ are adjacent in $\Gamma_{\mathcal{B}}$.
Let $b_1\in B_1$ and recall that   $|B_1|=|B_2|=2$.

Suppose first that $|\Gamma(b_1)\cap B_2|=2$.
Then as  $\Gamma$ is  $G$-arc-transitive, $|\Gamma(b_2)\cap B_1|=2$ for each $b_2\in B_2$, and so
$$[B_1\cup B_2]\cong \K_{2,2}.$$
Since $\Gamma_\mathcal{B}\cong  \K_n$, it follows that   $\Gamma\cong  \K_{n[2]}$. In particular, for each vertex $u$ of $\Gamma$, there is a unique vertex at distance two from $u$. Since $\Gamma$ is  $G$-arc-transitive, it follows that $G$ has rank three,
so that (2.1) holds.

Next assume that   $|\Gamma(b_1)\cap B_2|=1$.
First suppose that  $[B_1\cup B_2]$ contains exactly one edge.
Then  the valency of $\Gamma_\mathcal{B}$ is twice  the valency  of $\Gamma$. Since $\Gamma_\mathcal{B}\cong \K_n$,  the valency of $\Gamma_\mathcal{B}$ is $n-1$ and so $n$ is an odd integer. Let $B_1=\{b_1,b_1'\}$ and suppose that $b_1$ is adjacent to $b_2\in B_2$. Then $b_1'$ is not adjacent to any vertex of $B_2$. Furthermore, $b_1$ is adjacent to a unique vertex of $\frac{n-1}{2}$  neighbours in $\Gamma_{\mathcal{B}}$ of $B_1$, say $\Theta_1$, and $b_1'$ is adjacent to a unique vertex of the remaining $\frac{n-1}{2}$  neighbours in $\Gamma_{\mathcal{B}}$ of $B_1$, say $\Theta_2$.  Since  $G_{B_1,B_2}$ fixes $b_1$ and $b_2$,  it fixes   $\Theta_1$ and  $\Theta_2$ setwise.
Hence, as $n\geqslant 4$, it follows that  $G$ is not 3-transitive on $V(\Gamma_\mathcal{B})$.
Thus  Lemma \ref{primitive-cyclic-1} implies that   $\PGL(d,q)\leqslant G\leqslant \PGammaL(d,q)$ and $n=(q^d-1)/(q-1)$ with $d\geqslant 3$ (note that $\PGL(2,q)$
is 3-transitive on $q+1$ vertices).  Moreover, the elements of $V(\Gamma_{\mathcal{B}})$ can be identified with the set of 1-dimensional subspaces of a $d$-dimensional vector space over $\GF(q)$. Then  we see that $G_{B_1,B_2}$ has two orbits on the set of remaining 1-dimensional subspaces: those in the span $U$ of $B_1$ and $B_2$, and those outside $U$. Since $B_2\in\Theta_2$, it follows that $\Theta_1$ is one of these two orbits, but neither has size $\frac{n-1}{2}$,  a contradiction.

Thus  $[B_1\cup B_2]\cong 2\K_2$, and $\Gamma$ is a  cover of $\Gamma_\mathcal{B}\cong \K_n$.
It follows that for any vertex $v\in V(\Gamma)\setminus B_1$, $v$
is adjacent to either $b_1$ or $b_1'$, so $v\in \Gamma(b_1) \cup \Gamma(b_1')$. Thus
$$V(\Gamma)=\{b_1\}\cup \Gamma(b_1) \cup \Gamma(b_1') \cup \{b_1'\}.$$
Since $\Gamma$ is a cover of $\Gamma_{\mathcal{B}}$, we have that $\Gamma(b_1)\neq \Gamma(b_1')$ and since $\Gamma$ is $G$-arc-transitive, the vertex stabiliser $G_{b_1}=G_{b_1'}$ is transitive on $\Gamma(b_1)$ and $\Gamma(b_1')$. Thus $G_{b_1}$ has 4 orbits on $V(\Gamma)$.  Since $\Gamma$ is a  cover of $\Gamma_\mathcal{B}$, we have $b_1'\notin \Gamma_2(b_1)$.
Thus $\Gamma$  has diameter 3, $\Gamma_2(b_1)=\Gamma(b_1')$ and $\Gamma_3(b_1)=\{b_1'\}$. Hence $\Gamma$
is $G$-distance-transitive and is an antipodal cover of $\K_{n}$ with fibres of size 2.
Therefore $\Gamma$ and $G$ are  listed in cases (1), (3) or (6) of   \cite[Main Theorem]{GLP}.
Recall that either $(G,n)=(M_{11},11)$, or  $n$ and $G$ as an abstract group are given in
  cases (ii)--(iii) of Lemma \ref{primitive-cyclic-1}. Thus, either case  (1) or case (3)(e)
of   \cite[Main Theorem]{GLP} holds. If case  (1)
 occurs then $\Gamma\cong \K_{n,n}-n\K_2$. However, $M_{11}$ does not have a rank four action of degree 22 while $S_n$ and $A_n$ do not have a quasiprimitive action of degree $2n$ where the blocks have size two.
Thus $(G,n)$ is as in case (iii) of Lemma \ref{primitive-cyclic-1}, that is, $\PGL(d,q)\leqslant G\leqslant \PGammaL(d,q)$,  and $n=(q^d-1)/(q-1)$, so that (2.2) holds.
Finally, if case  (3)(e)
 occurs, then
$\Aut(\Gamma)\cong \PSigmaL(2,p)\times S_2$ where  $p=n-1\equiv 1\pmod{4}$.
Since $G$  is quasiprimitive  on $V(\Gamma)$ and is transitive on the set of arcs of $\K_{n}$, it follows that $G$ is isomorphic to a subgroup of $ \PSigmaL(2,p)$. However, here $G$ does not contain an element of order $n=p+1$.
This completes the proof. \end{proof}

We remark that  quasiprimitive rank 3 groups have been classified in  \cite{DGLPP}, and  the possibilities for the group $G$
in case (2.1) of Proposition \ref{bicirculant-quasiprimitive-th1} can be determined by using
\cite[Table 1]{DGLPP}.

\section{Vertex bi-quasiprimitive arc-transitive bicirculants}

In this section, we  will complete the proof
of  Theorem \ref{bicirculant-reduction-th2} by  determining the arc-transitive  vertex bi-quasiprimitive bicirculants.

Let $\Gamma$ be a $G$-arc-transitive   graph and suppose that  $G$ acts  bi-quasiprimitively  on  $V(\Gamma)$. Then $G$ has a minimal normal subgroup $M$ that has exactly two orbits on $V(\Gamma)$.
Since  $\Gamma$ is $G$-arc-transitive and connected,
each $M$-orbit contains no edge of $\Gamma$.
Thus $\Gamma$ is a bipartite graph, and  the two $M$ orbits form the two bipartite halves of $\Gamma$.
In particular,   all intransitive  normal subgroups of $G$ have the same orbits. Recall that $G^+$ denotes the index two subgroup of $G$ that is the stabiliser of each bipartite half.

The following proposition  classifies arc-transitive bicirculants that are vertex bi-quasiprimitive and such that the two orbits of the cyclic subgroup are the two bipartite halves.

\begin{prop}\label{bicirc-biqp-lem2}
Let $\Gamma$ be a connected $G$-arc-transitive graph  such that  $G$ is bi-quasiprimitive on $V(\Gamma)$ and $G$ contains a cyclic subgroup $H$ of order $n$ such that the two bipartite halves are $H$-orbits.
Then  $\Gamma$ is one of the following graphs:
\begin{enumerate}[{\rm (a)}]
\item $\K_{n,n}$ where $n\geqslant 2$;
\item $ \K_{n,n}-n\K_2$ where $n\geqslant 3$;
\item  $B'(H(11))$;
\item $G(2p,r)$ where  $p$ is a prime and  $r>1$ divides $p-1$;
\item $B(\PG(d-1,q))$ and $B'(\PG(d-1,q))$, where $d\geq 3$, $q$ is a prime power.
\end{enumerate}

\end{prop}
\begin{proof}   Identify the two bipartite halves of $\Gamma$ with $H$ and denote them by $H_0$ and $H_1$.
Let $G^+=G_{H_0}=G_{H_1}$.
Since $G$ is transitive on the vertex set, it follows that  $G=\langle G^+, \sigma \rangle$	for some element $\sigma\in G\setminus G^+$ with $\sigma^2\in G^+$.
If   $G^+$ acts unfaithfully on each $H_i$,
then it follows from Lemma \ref{lem:Knn} that $\Gamma=\K_{n,n}$ for some  $n\geqslant 2$.

From now on we suppose that  $G^+$ acts faithfully on each $H_i$.
Assume that $G^+$ acts  imprimitively on $H_0$. Take a maximal $G^+$-invariant partition $\mathcal{B}_0$ on $H_0$. Then $\mathcal{B}_0^\sigma$ is a maximal $G^+$-invariant partition  on $H_1$.
Let $\mathcal{C}=\mathcal{B}_0 \cup \mathcal{B}_0^\sigma$. Then $G$ leaves invariant this partition
$\mathcal{C}$ of the vertex set, as $\mathcal{B}_0^{\sigma^2}=\mathcal{B}_0$. Since  $G$ is bi-quasiprimitive on $V(\Gamma)$ and
$|\mathcal{C}|\geqslant 4$, it follows that $G$  acts faithfully on $\mathcal{C}$.
Let $M_0$ be the kernel of $H$ acting on $\mathcal{B}_0$ and $M_1$ be the kernel of $H$ acting on $\mathcal{B}_0^\sigma$.
Since $H$ acts transitively on $\mathcal{B}_0$ and  $\mathcal{B}_0^\sigma$, it follows that  $|M_0|= |H|/|\mathcal{B}_0|=|H|/|\mathcal{B}_0^{\sigma}|=|M_1|$. Then as $H$ is a cyclic group and has a unique subgroup of each order, it follows that
$M_0=M_1$, that is, the kernel of $H$ on $\mathcal{B}_0$ is the same as the kernel on $\mathcal{B}_0^\sigma$, and is hence in the kernel of $G$ on $\mathcal{C}$. However, $G$ is faithful on $\mathcal{C}$, and so $H$ acts faithfully on $\mathcal{B}_0$. Thus $|\mathcal{B}_0|=|H|$, contradicting $G^+$ being imprimitive on $H_0$.

Therefore  $G^+$ is primitive on each $H_i$.
Since $G^+$ contains the cyclic subgroup $H$ that is transitive on each $G^+$-orbit, it follows from Lemma \ref{primitive-cyclic-1} that
either $|H|=p$ is a prime or $G^+$ is $2$-transitive on $H_i$.
If  $|H|=p$, then as $\Gamma$ is a bipartite graph  it follows from Theorem \ref{bicirc-2p-1} that
$\Gamma$ is one of the following graphs: $\K_{p,p}$, $G(2p,r)$ with $r>1$, $B'(H(11))$ where $p=11$, $B(\PG(d-1,q))$ and $B'(\PG(d-1,q))$ where $p=\frac{q^d-1}{q-1}$,  $d\geqslant 3$ and $q$ is a prime power.
Suppose that $G^+$ is $2$-transitive on $H_i$
with $|H_i|$ not a prime.  Then Lemma \ref{lem:2trans} (1) implies that
$\Gamma=\K_{n,n}-n\K_2$ where $n\geqslant 3$, $B(\PG(d-1,q))$, or $B'(\PG(d-1,q))$.
\end{proof}

It remains to  consider the case  where the two $H$-orbits are  not the two bipartite halves.

\begin{prop}\label{bicir-biqp-2}
Let $\Gamma$ be a $G$-arc-transitive graph such that $G$ is bi-quasiprimitive on $V(\Gamma)$ and $G$ contains a cyclic subgroup  $H$  of order $n$, such that $H$ has two orbits of size $n$ and these are not the bipartite halves of $\Gamma$.
Then $n$ is even and  $\Gamma$ is one of $\K_{n,n}$, $\K_{n,n}-n\K_2$, $B(\PG(d-1,q)$ or $B'(\PG(d-1,q))$ for even $d\geqslant 3$ and odd $q$.

\end{prop}
\begin{proof}
Let $\Delta_0$ and $\Delta_1$ be the two bipartite halves of $\Gamma$, and let $H_0$ and $H_1$ be the two $H$-orbits. Let $G^+$ be the index two subgroup of $G$ that stabilises both $\Delta_0$ and $\Delta_1$, and let $\sigma\in G$ such that $G=\langle G^+,\sigma\rangle$ and  $\sigma^2\in G^+$. If $|V(\Gamma)|\leqslant 4$, then the only candidate for $\Gamma$ is
$ \K_{2,2}$, so from now on we suppose that $|V(\Gamma)|>4$. Then $\{\Delta_0,\Delta_1\}$ is the unique $G$-invariant partition of  $V(\Gamma)$ into two equal sized parts.  Since the $\Delta_i$ are not $H$-orbits and $H$ has two orbits of size $n$, it follows that $H^+:=H\cap G^+$ has index two in $H$ and has two equal sized orbits on each $\Delta_i$. Thus $n$ is even and $G=\langle G^+,H\rangle$.

If $G^+$ is unfaithful on each $\Delta_i$, then it follows from Lemma \ref{lem:Knn} that $\Gamma=\K_{n,n}$.  Thus in the remainder we assume that $G^+$ is faithful on each $\Delta_i$.

Suppose first that $G^+$ is primitive on each $\Delta_i$. Since $H^+$ is a cyclic subgroup with exactly two orbits of size $n/2$ on $\Delta_i$, the possibilities for $G^+$ are given by Theorem \ref{bicirculant-primitive-2}. In particular, either  $G^+$ is almost simple or $\mathbb{Z}_2^m \lhd G^+\leqslant \AGL(m,2)$ with $2\leqslant m\leqslant 4$.  If $\mathbb{Z}_2^m \lhd G^+\leqslant \AGL(m,2)$ then $G^+$ has a normal subgroup $N\cong \mathbb{Z}_2^m$ that is regular on each $\Delta_i$.  Now $N$ is characteristic in $G^+$ and so is normal in $G$.  Moreover, since $H$ is cyclic, either $H\cap N=1$ or $H\cap N=\mathbb{Z}_2$.  Further, since $N$ is self-centralising in $G^+$ and $H$ is cyclic, either $|H|=4$ and $C_H(N)=H$, or $C_H(N)= H^+\cap N$.  If $m\geqslant 3$, then $|H|=2^m\geqslant 8$ and  $H^+\cap N\neq H^+$,  so $C_H(N)=H^+\cap N$. Thus $H/(H^+\cap N)$ is isomorphic to a  cyclic  subgroup of $\GL(m,2)$ of order $2^m$ or $2^{m-1}$. Since $m=3$ or $4$, it follows that $m=3$, $G^+=\AGL(3,2)$ and $H^+\cap N=2$. Since $|H|=8$ and $G^+$ does not have an element of order 8, it follows that $G=\Aut(\AGL(3,2))$. In particular, given $u\in \Delta_0$ we have that $G_u^+$ is transitive on $\Delta_1$ and so $\Gamma=\K_{8,8}$.  It remains to consider the case where $m=2$ and $G^+=A_4\cong \AGL(1,4)$ or $S_4\cong \AGL(2,2)$. If $G^+=S_4$ then $G=S_4\times \mathbb{Z}_2$, contradicting $G$ being bi-quasiprimitive. Thus $G^+=A_4$ and either $G=A_4\times \mathbb{Z}_2$ or $G\cong S_4$. The first group does not contain an element of order 4 (and is also not bi-quasiprimitive), while the second implies that $\Gamma=\K_{4,4}-4\K_2$.

Next suppose that $G^+$ is almost simple. Then either $G^+$ is $2$-transitive on each $\Delta_i$, or $n=10$ and $A_5\leqslant G^+\leqslant S_5$. In the first case, Lemma \ref{lem:2trans} (2) implies that $\Gamma=\K_{n,n}-n\K_2$
where $n\geqslant 3$,  $\K_{12,12}$, $B(\PG(d-1,q))$ or $B'(\PG(d-1,q))$ where  $d\geqslant 3$ is an even integer and $q$ is odd.
Suppose instead that $n=10$ and $A_5\leqslant G^+\leqslant S_5$. Then by \cite[p.75]{DM-1}, the action of $G^+$ on $\Delta_i$ is $S_5$ or $A_5$ acting naturally on the set of unordered pairs of $\{1,2,3,4,5\}$.
Let $u\in \Delta_0$.
Then  $G_{u}^+$ has only orbits of size 1, 3 or 6 on both $\Delta_0$ and $\Delta_1$.
Moreover, there exists $u'\in \Delta_1$ such that
$G_{u}^+=G_{u'}^+=G_{uu'}^+$.
Since $G_u^+$ is transitive on $\Gamma(u)$,
it follows that $|\Gamma(u)|=3$ or 6. Thus $\Gamma$ is either the standard double cover of the Petersen graph, or the standard double cover of the complement of the Petersen graph. In both cases $\Aut(\Gamma)=S_5\times \mathbb{Z}_2$. If $G\leqslant \Aut(\Gamma)$ contains an element $g$ of order 10 with two cycles of length 10 on $V(\Gamma)$, it follows that $g^5$ is an involution that centralises the element $g^2\in S_5$ of order five. As $S_5$ contains no such involution, we have  $\langle g^5\rangle =Z(\Aut(\Gamma)) \leqslant G$, contradicting the fact that all normal subgroups of $G$ have at most two orbits.

It remains to consider the case where $G^+$ is imprimitive on $\Delta_i$. Let $\mathcal{B}_0$ be a maximal $G^+$-invariant partition of $\Delta_0$.
Then $\mathcal{B}_1=\mathcal{B}_0^\sigma$ is a maximal $G^+$-invariant partition of $\Delta_1$.  Note that  $\mathcal{B}=\mathcal{B}_0\cup \mathcal{B}_1$
is a  $G$-invariant partition of $V(\Gamma)$ with at least four parts. Thus the kernel of $G$ acting on $\mathcal{B}$ has at least four orbits and so the bi-quasiprimitivity of $G$ implies that $G$ acts faithfully on $\mathcal{B}$.

Suppose that each  $B\in \mathcal{B}$ is  contained in either $H_0$ or $H_1$. Then for $i\in \{0,1\}$, let $\mathcal{C}_i=\{B\in \mathcal{B}\mid B\subseteq H_i\}$ and let $M_i$ be the kernel of $H$ acting on $\mathcal{C}_i$. Then $H/M_i$ is transitive and so regular on $\mathcal{C}_i$. Since $|\mathcal{C}_0|=|\mathcal{C}_1|$, we have $|M_0|=|M_1|$,
and since $H$ is cyclic, it follows that $M_0=M_1$. Thus $M_0$ lies in the kernel of $G$ acting on $\mathcal{B}$. Since $G$ is faithful on $\mathcal{B}$ it follows that $M_0=1$. Hence $|\mathcal{B}|=2|H|=|\Omega|$, contradicting the fact that $\mathcal{B}$ is a non-trivial block system of $G$ on $V(\Gamma)$. Thus Lemma \ref{quo-1} (1) implies that
each $B\in \mathcal{B}$  has non-empty intersection with both $H_0$ and $H_1$.  Hence $H$ is transitive and faithful on $\mathcal{B}$. Since $H$ is cyclic it follows that $H$ is regular on $\mathcal{B}$, so $\Gamma_{\mathcal{B}}$ is a circulant.
As $|H|=|V(\Gamma)|/2$, it follows that  each $B\in\mathcal{B}$ has size two. Moreover, $H^+$ acts transitively on each $\mathcal{B}_i$.

For $i=0,1,$ let $K_i$ be the kernel of $G^+$ acting on $\mathcal{B}_i$. Then $K_0^\sigma=K_1$, and   $K_0\cap K_1$ is contained in the kernel of $G$ acting on $\mathcal{B}$. Thus  $K_0\cap K_1=1$ and so $K_0\times K_1\lhd G$. Since $G$ is bi-quasiprimitive, it follows that $K_i=1$ or $K_i$ acts transitively on $\mathcal{B}_{1-i}$.

Suppose that $K_i$ acts transitively on $\mathcal{B}_{1-i}$. Then $\Gamma_\mathcal{B}\cong \K_{\frac{n}{2},\frac{n}{2}}$ is complete bipartite.
Since  $K_0\times K_1\unlhd G^+$, it follows that  $(K_0\times K_1)/K_0\unlhd G^+/K_0\cong (G^+)^{\mathcal{B}_0}$.
Since each block in $\mathcal{B}_i$ has size 2, it follows that $K_i$ contains only elements of order 2. Thus $K_i$ is abelian and $K_i\cong \mathbb{Z}_{2}^{r}$. Based on this together with the fact that  $K_i$ acts transitively on $\mathcal{B}_{1-i}$, we conclude that  $K_i$ is regular  on $\mathcal{B}_{1-i}$, and so $\frac{n}{2}=|\mathcal{B}_{1-i}|=|K_i|=2^r$.
Furthermore,   $ \mathbb{Z}_2^{r}\cong (K_0\times K_1)/K_0  \unlhd  (G^+)^{\mathcal{B}_0}$.
Since  $\mathcal{B}_0$ is a maximal $G^+$-invariant partition of $\Delta_0$, $(G^+)^{\mathcal{B}_0}$ acts primitively on $\mathcal{B}_0$.
In particular,  $(G^+)^{\mathcal{B}_0}$ acts primitively on $\mathcal{B}_0$ of affine type.
Recall that the cyclic group $(H^+)^{\mathcal{B}_0}\leqslant (G^+)^{\mathcal{B}_0}$ acts transitively on  $\mathcal{B}_0$.
By Lemma \ref{primitive-cyclic-1}, $\frac{n}{2}=p$ is a prime, so  $\frac{n}{2}=p=2$. Hence $n=4$ and $\Gamma_\mathcal{B}\cong \K_{2,2}$.
Let $\mathcal{B}_0=\{B_1,B_2\}$ and $\mathcal{B}_1=\{D_1,D_2\}$.
If  $[B_i\cup D_j]$ contains exactly one edge of $\Gamma$, then $\Gamma$ has valency 1, and so $\Gamma$ is disconnected, a contradiction.
Suppose that $[B_1\cup D_1]\cong 2\K_2$. Then $\Gamma$ is a  cover of $\Gamma_\mathcal{B}$. Hence $\Gamma$ has valency 2 and $\Gamma$ is a cycle. Since
each block in $\mathcal{B}$ has size 2, it follows that $\Gamma\cong C_8$.
However, in this case, $\Aut(\Gamma)$ has a unique order 4 cyclic subgroup whose 2 orbits are the 2 bipartite halves of $\Gamma$, a contradiction.
If $[B_1\cup D_1]\cong \K_{2,2}$, then $\Gamma\cong \K_{4,4}$.

Thus it remains to consider the case that $G^+$ acts faithfully on each $\mathcal{B}_i$. The maximality of $\mathcal{B}_i$ implies that $G^+$  is primitive on $\mathcal{B}_i$.
If $|\mathcal{B}_i|=p$ and $G^+\leqslant \AGL(1,p)$, then $G^+$ has a unique minimal normal subgroup $N$ of order $p$. Since $N$ is characteristic in $G^+$, it is normal in $G$. However, $|V(\Gamma)|=4p$ and so $N$ has four orbits, contradicting $G$ being bi-quasiprimitive. Since $H^+$ is a cyclic transitive subgroup of $G^+$ in its action on $\mathcal{B}_i$, it follows that $G^+$ is given by Lemma  \ref{primitive-cyclic-1} and in particular is 2-transitive on $\mathcal{B}_i$. Moreover, if $B\in\mathcal{B}_0$ then $(G^+)_B$ has an index two subgroup (the stabiliser of a vertex in $B$) and so either $G^+=S_{n/2},M_{11}$ or $\PGL(d,q)\lhd G^+\leqslant \PGammaL(d,q)$.  If $G^+=M_{11}$ or $S_{n/2}$ with $n\neq 12$, then  $G=G^+ \times \mathbb{Z}_2$, contradicting $G$ being bi-quasiprimitive.
If  $G^+=S_{6}$, then $G=\Aut(S_6)$. However, in this case, all the order 6 elements of $G$ are in $G^+=S_6$, again a contradiction.
Thus $\PGL(d,q)\lhd G^+\leqslant \PGammaL(d,q)$. Then either $|C_G(\PSL(d,q))|=2$ or $G\leqslant \Aut(\PSL(d,q))$. The first case is not possible as this would imply that $G$ has a normal subgroup of order two, which would contradict $G$ being bi-quasiprimitive. Thus $G\leqslant \Aut(\PSL(d,q))$. If $G\not\leqslant\PGammaL(d,q)$ then $\Gamma_\mathcal{B}$ is either   $B(\PG(d-1,q))$ or $B'(\PG(d-1,q))$. However, these graphs are not circulants (Lemma \ref{basiccirc-1}) and so $G\leqslant \PGammaL(d,q)$. In this case  $|\mathcal{B}|=2(q^d-1)/(q-1)$, but $\PGammaL(d,q)$ does not contain an element of this order, contradicting $H$ being regular on $\mathcal{B}$.
 \end{proof}

\section{Prime and small valency arc-transitive bicirculants}

In this section, we compare our Theorem \ref{bicirculant-reduction-th2} to the classifications of arc-transitive bicirculants of valency 3 and 5, and also make some observations about
the prime valent case in general.

\medskip

We  first give a corollary of Theorem \ref{bicirculant-reduction-th2} about prime valency arc-transitive bicirculants.

\begin{corol}\label{th-cor-1}
Let $\Gamma$ be a  $G$-arc-transitive  bicirculant of prime valency $p\geq 3$. Then
 $\Gamma$ is a normal cover of
one of the following graphs.
 \begin{enumerate}[{\rm (a)}]
\item    $\K_{p,p}$;
\item   $\K_{p+1}$,  $\K_{p+1,p+1}-(p+1)\K_2$;
\item   $G(2q,p)$, where $q$ is a   prime integer   and  $p$ divides    $q-1$;
\item   $B(\PG(n,q))$,  where $q$ is a prime power,  $n\geqslant 2$, and $p=\frac{q^n-1}{q-1}$;
\item Petersen graph ($p=3$),  Clebsch graph ($p=5$).
 \end{enumerate}
Moreover, examples exist in all cases.
\end{corol}
\begin{proof}
Let $\Gamma$ be a  $G$-arc-transitive  bicirculant of prime valency $p\geqslant 3$. Then by Theorem \ref{bicirculant-reduction-th2}, $G$ has a normal subgroup $N$ such that  $\Gamma$
is a normal $r$-cover of $\Gamma_N$ where $\Gamma_N$ is
one of the    graphs in Theorem \ref{bicirculant-reduction-th2}. Moreover, since $r$ divides the prime number $p$, it follows that $r=1$, and so
$\Gamma$ is a normal cover of $\Gamma_N$. Thus $\Gamma_N$ has valency $p$.

Note  that, for the    graphs in Theorem \ref{bicirculant-reduction-th2},   $\K_{n[2]}$ has valency $2(n-1)$ which is not a prime; $\Cay(q,e)$ ($q$  a prime) has even valency $e$;
$B'(\PG(n,q))$ has valency $q^n$; $\Ha(2,4)$, $B'(H(11))$ and  the complement of the Petersen graph have valency 6;  and the complement of the Clebsch graph has valency 10.
Thus
$\Gamma_N$ is as claimed. Moreover, since all  graphs listed in (a)-(e) are themselves arc-transitive prime valency bicirculants by Lemma \ref{basiccirc-3}, examples trivially exist for all these listed graphs.
\end{proof}

For the  class of arc-transitive bicirculants  of valency 3 or 5  that are Cayley graphs of dihedral groups we obtain the following lemma by using  \cite{AHK-2015} and \cite{MP-2000}.

\begin{lemma}\label{dihval35-1}
Let $n\geqslant 11$ and $k= 3$ or $5$. Let $\Gamma_{n,k}=\Cay(D_{2n},\{b,ba,ba^{r+1},\ldots,ba^{r^{k-2}+\cdots+r+1}\})$, where   $D_{2n}=\langle a,b|a^n=b^2= (ba)^2=1 \rangle$, and  $r \in \mathbb{Z}_n^*$ such that  $r^{k-1}+\cdots+r^2+r+1 \equiv 0 \pmod n$.
Then   $\Gamma_{n,k}$ is a bipartite  arc-transitive graph  and  the following hold.

 \begin{enumerate}[{\rm (a)}]
\item  $\langle a^p  \rangle\cong \mathbb{Z}_{n/p}$ is a normal subgroup of $\Aut(\Gamma_{n,k})$, where $p$ is a prime divisor of $n$.
\item   There exists  a prime divisor $p$ of $n$ such that $k \mid (p-1)$.
\item  Let  $p$ be  a prime divisor of $n$ such that $k\mid (p-1)$ and let $N=\langle a^p  \rangle$. Then $(\Gamma_{n,k})_N\cong G(2p,k)$.
 \end{enumerate}

\end{lemma}
\begin{proof}
By the definition of $\Gamma_{n,k}$, we  see that $\Gamma_{n,k}$ is a bipartite graph, and the two bipartite halves of $\Gamma_{n,k}$  are  $\Delta_0=\{1,a,a^2,\ldots, a^{n-1}\}$ and $\Delta_1=\{b,ba,ba^2,\ldots, ba^{n-1}\}$
which are both $\langle a\rangle$-orbits.

If $k=3$, then by \cite[p.978--979]{MP-2000}, $\Gamma_{n,3}$ is arc-transitive  with   $\Aut(\Gamma_{n,3})=\mathbb{Z}_n:\mathbb{Z}_6$ where $\mathbb{Z}_n=\langle a  \rangle$;
if $k=5$, then by Lemma 3.7 and the proof of Theorem 3.11 of \cite{AHK-2015}, $\Gamma_{n,5}$ is arc-transitive and
$\langle a\rangle\cong \mathbb{Z}_n$ is a normal subgroup of   $\Aut(\Gamma_{n,5})$.
Thus in both cases $\langle a\rangle\cong \mathbb{Z}_n$ is a normal subgroup of   $\Aut(\Gamma_{n,k})$.
Since for each   prime divisor $p$ of $n$,  $\langle a^p  \rangle$ is a characteristic subgroup of $\langle a  \rangle$,
it follows that   $\langle a^p  \rangle \cong \mathbb{Z}_{n/p}$ is a normal subgroup of $\Aut(\Gamma_{n,k})$.

Let $n=p_1^{e_1}p_2^{e_2}\cdots p_f^{e_f}$ where $p_i$ is a prime. Since $k<n$, we know $r\neq 1$.
Since $r^{k-1}+\cdots+r^2+r+1 \equiv 0 \pmod n$, it follows that  $(r-1)(r^{k-1}+\cdots+r^2+r+1) \equiv 0 \pmod n$, so
$r^k -1 \equiv 0 \pmod n$, that is, $r^k \equiv 1 \pmod n$.
Let $(n,r)=t$. Since $r^k=ne+1$ for some integer $e$, it follows that $t$ divides $1$, so
 $t=1$.
Hence  $r$ belongs to the group of units of the ring $\mathbb{Z}_n$, which has order $\phi(n)$ (Euler totient function).
Therefore, since $k$ is a prime, the order of $r$ in the group of units is $k$, and so  $k$  divides $\phi(n)=p_1^{e_1-1}p_2^{e_2-1}\cdots p_f^{e_f-1} \Pi (p_i-1)$.
Suppose that $k$ (which is a prime) divides $p_1^{e_1-1}p_2^{e_2-1}\cdots p_f^{e_f-1}$. Then  $k^2$ divides $n$. As  $r^{k-1}+\cdots +r^2+r+1 \equiv 0 \pmod n$, we have $r^{k-1}+\cdots +r^2+r+1 \equiv 0 \pmod {k^2}$.
However, considering the possibilities for $r$ module $k^2$ and evaluating $r^{k-1}+\cdots+r^2+r+1$, we see that we never have $r^{k-1}+\cdots+r^2+r+1\equiv 0\pmod {k^2}$,
a contradiction.
Thus  $k^2$ does not divide $n$, and so $k$ divides $ \Pi (p_i-1)$. Since $k$ is a prime number,  there exists a prime divisor $p_i$ of $n$ such that $k$ divides $p_i-1$.

Let  $p$ be a prime divisor of $n$ such that $k\mid (p-1)$ and let $N=\langle a^p  \rangle$. Then $N$ has $p$ orbits on each bipartite half.
Thus $(\Gamma_{n,k})_N$ is a bipartite arc-transitive graph of  valency $k$ and with $2p$ vertices, and also its automorphism group has two blocks of size $p$.
By \cite[Lemma 3.9]{CO-1987}, $(\Gamma_{n,k})_N\cong G(2p,k)$.
\end{proof}

\subsection{Arc-transitive bicirculants of valency 3}

The \emph{generalised Petersen graph} $GP(n,r)$  is the graph on the vertex set
$$\{u_0,u_1,\ldots,u_{n-1},v_0,v_1,\ldots,v_{n-1}\}$$
with the adjacencies:
$$u_i\sim u_{i+1}, \, v_i \sim v_{i+r}, \, u_i\sim v_i, \quad\quad i=0,1,\ldots,n-1.$$
Hence each  generalised Petersen graph $GP(n,r)$ has valency 3.
By \cite{FGW-1971},
$GP(n,r)$ is arc-transitive if and only if   $(n,r)=(4,1)$, $(5, 2) $, $(8, 3)$, $ (10, 2),$ $ (10, 3), (12, 5) $  and $  (24, 5)$.

\medskip

The following classification of  cubic arc-transitive bicirculants follows from  \cite{FGW-1971,MP-2000,P-2007}.

\begin{theo}\label{atbic-val3-1}
Let $\Gamma$ be a finite connected arc-transitive  bicirculant of valency $3$. Then   $\Gamma$ is one of the following graphs:

 \begin{enumerate}[{\rm (a)}]
\item    $\K_4$, $\K_{3,3}$;
\item   $GP(4,1)$ ($\K_{4,4}-4\K_2$), $GP(5,2)$ (Petersen graph), $GP(8,3)$ (M\"obius-Kantor graph), $GP(10,2)$, $GP(10,3)$ (Desargues graph), $GP(12,5)$ (Nauru graph), $GP(24,5)$;
\item    Heawood graph ($B(\PG(2,2))$);
\item    $\Cay(D_{2n},\{b,ba,ba^{r+1}\})$, where $D_{2n}=\langle a,b|a^n=b^2=(ba)^2=1 \rangle$,  $n\geqslant 11$ is odd and $r\in \mathbb{Z}_{n}^*$ such that $r^2+r+1 \equiv 0 \pmod n$.
 \end{enumerate}

\end{theo}

Define the following permutations of $V(GP(n,r))$.
\begin{align*}
\rho: & \  u_i\mapsto u_{i+1},v_i\mapsto v_{i+1};\\
\delta: & \  u_i\mapsto u_{-i},v_i\mapsto v_{-i};\\
\sigma: & \  u_{4i}\mapsto u_{4i},u_{4i+2}\mapsto v_{4i-1}, u_{4i+1}\mapsto u_{4i-1},u_{4i-1}\mapsto v_{4i},\\
         & \ v_{4i}\mapsto u_{4i+1},v_{4i-1}\mapsto v_{4i+5}, v_{4i+1}\mapsto u_{4i-2},v_{4i+2}\mapsto v_{4i-6}.
\end{align*}

The following lemma essentially follows from \cite{FGW-1971}.
\begin{lemma}\label{gpet-1}
Let $\Gamma=GP(n,r)$ where   $(n,r)=(4,1),$ $(8, 3), $ $ (12, 5) $ or $  (24, 5)$.
Let $A:=\Aut(\Gamma)$. Then $\Gamma$ is bipartite,  $2$-arc-transitive  and the following hold.

 \begin{enumerate}[{\rm (a)}]
\item    $A=\langle \rho, \sigma, \delta \rangle$ and $A_{u_0}=\langle \sigma, \delta \rangle$.
\item   $\langle \rho^4  \rangle\cong \mathbb{Z}_{n/4}$ is a normal subgroup of $\Aut(\Gamma)$.
\item  If $N=\langle \rho^4  \rangle$, then $\Gamma_N\cong \K_{4,4}-4\K_2$.
 \end{enumerate}

\end{lemma}
\begin{proof}
By \cite{HPZ-1}, the generalised Petersen graph $GP(n, r)$ is bipartite if and only if $n$ is even and
$r$ is odd, so  $\Gamma$ is bipartite.
By \cite[p.217]{FGW-1971},  $\Gamma$ is   $2$-arc-transitive, and  $A=\langle \rho, \sigma, \delta \rangle$, where   $\rho^{n}=\delta^2=\sigma^3=1$, $\delta \rho \delta =\rho^{-1}$, $\delta \sigma \delta =\sigma^{-1}$, $\sigma \rho \sigma =\rho^{-1}$, $\sigma \rho^4=\rho^4 \sigma$. Thus   $\langle \rho^4  \rangle$ is  a normal subgroup of $A$.
Moreover, by definition,  $A_{u_0}=\langle \sigma, \delta \rangle$.

Note that the two orbits $H_0$ and $H_1$ of $\langle \rho \rangle$ are not the two bipartite halves  of $\Gamma$.
Thus $H_i$ intersects each bipartite half with exactly $n/2$ vertices. Further, $\langle \rho^2 \rangle$ fixes the two bipartite halves of $\Gamma$ setwise.
Let   $N=\langle \rho^4  \rangle$. Then   each $N$-orbit is in a bipartite half,
$\Gamma_N$ has 8 vertices  and has   valency  3, and
it is also 2-arc-transitive. It follows that  $\Gamma_N\cong \K_{4,4}-4\K_2$.
\end{proof}

By \cite[p.217]{FGW-1971}, $GP(10,2)$ has automorphism group $A_5\times S_2$, and   $GP(10,3)$ is the Desargues graph which has automorphism group $S_5\times S_2$.
For these two graphs, choose the normal subgroup $S_2$ of the automorphism group, then the  corresponding quotient graph has 10 vertices and is of valency 3. Moreover, the automorphism group induces
$A_5$ or $S_5$ on the quotient graph, respectively,  and in both cases the quotient graph is the Petersen graph.
The graphs    $\K_4$, $\K_{3,3}$, $GP(4,1)=\K_{4,4}-4\K_2$, $GP(5,2)$ and $G(14,3)$ are in Theorem \ref{bicirculant-reduction-th2}.

Lemmas \ref{dihval35-1} and \ref{gpet-1} now allow us to show how each cubic arc-transitive bicirculant gives rise to a graph in Theorem \ref{bicirculant-reduction-th2}.

\begin{prop}\label{bicirculant-val3}
Let $\Gamma$ be a finite connected arc-transitive  bicirculant of valency $3$.
Then there exists a normal subgroup $N$ of $\Aut(\Gamma)$ such that $\Gamma_N$ is a graph in Theorem \ref{bicirculant-reduction-th2}.
Moreover,  $\Gamma$, $\Gamma_N$, $\Aut(\Gamma)$ (if known), and $N$  are as in Table \ref{table-val3}.
\end{prop}

\begin{table}[h]\caption{ Cubic arc-transitive bicirculants }\label{table-val3}
\begin{tabular}{|c|c|c|c|c|}
\hline
 $\Gamma$ &  $\Gamma_{N}$ & $\Aut(\Gamma)$ &  $N$       \\
\hline
$\K_4$ &  $\K_4$  & $S_4$ & $1$   \\
\hline
$\K_{3,3}$ & $\K_{3,3}$ & $S_3\wr S_2$ &  $1$           \\
\hline
$\K_{4,4}-4\K_2$ & $\K_{4,4}-4\K_2$ &  $S_4\times S_2$  & $1$           \\
\hline
Heawood graph & Heawood graph  & $\PGL(3,2)$    & $1$   \\
\hline
$GP(5,2)$   & Petersen graph  &  $S_5$   & $1$      \\
 \hline
$GP(8,3)$ & $\K_{4,4}-4\K_2$  &  $GL(2,3):S_2$   & $S_2$    \\
 \hline
$GP(10,2)$  & Petersen graph  & $A_5\times S_2$  & $S_2$       \\
 \hline
$GP(10,3)$  & Petersen graph  &  $S_5\times S_2$   & $S_2$    \\
  \hline
$GP(12,5)$  & $\K_{4,4}-4\K_2$   & $S_4\times S_3$ & $\mathbb{Z}_3$        \\
 \hline
 $GP(24,5)$ & $\K_{4,4}-4\K_2$  & $|\Aut(\Gamma)|=288$    & $\mathbb{Z}_6$     \\
\hline
$\Cay(D_{2n},\{b,ba,ba^{r+1}\})$, & $G(2p,3)$,   & $\mathbb{Z}_n:\mathbb{Z}_6$ &   $\mathbb{Z}_{n/p}$  \\
$n\geqslant 11$ odd, as in Theorem \ref{atbic-val3-1} &  prime $p|n$, $3|p-1$  &    &      \\
\hline
\end{tabular}
\end{table}

Note that in Proposition \ref{bicirculant-val3}, $N$ and $\Gamma_N$ are not necessarily unique.
For instance, if $\Gamma$ is $GP(12,5)$, then its automorphism group is $S_4\times S_3$, and
$\Gamma_N$ is  $\K_4$ whenever $N=S_3$, and $\Gamma_N$ is   $\K_{3,3}$ when $N=\mathbb{Z}_2\times \mathbb{Z}_2$.

\subsection{Arc-transitive bicirculants of valency 5}
\bigskip

Let $L,R$ and $M$ be subsets of an additive group $H:=\mathbb{Z}_n$ such that  $L=-L$, $R=-R$ and $R\cup L$ does not contain the zero  element of $H$. Define the bicirculant  $BC_n[L,M,R]$
to have vertex set the union of the left part $H_0=\{h_0|h\in H\}$ and the right part $H_1=\{h_1|h\in H\}$, and edge set the union of the left edges $\{\{h_0,(h+l)_0\}|l\in L\}$, the
right edges $\{\{h_1,(h+r)_1\}|r\in R\}$ and the \emph{spokes} $\{\{h_0,(h+m)_1\}|m \in M\}$.

\medskip
The family of arc-transitive bicirculants of valency 5 have been classified in  \cite{AHK-2015,AHKOS-2015}.

\begin{theo}\label{atbic-val5-1}
Let $\Gamma$ be a finite connected arc-transitive  bicirculant of valency $5$. Then   $\Gamma$ is one of the following graphs:

 \begin{enumerate}[{\rm (a)}]
\item    $\K_6$;
\item     $\K_{6,6}-6\K_2$;
\item     $B(\PG(2,4))$;
\item   Clebsch graph;
\item     $BC_6[\{\pm 1\},\{0,1,5\},\{\pm 2\}]$;
\item     $BC_{12}[\varnothing,\{0,1,2,4,9\},\varnothing]$;
\item     $BC_{24}[\varnothing,\{0,1,3,11,20\},\varnothing]$;
\item     $\Cay(D_{2n},\{b,ba,ba^{r+1},ba^{r^2+r+1},ba^{r^3+r^2+r+1}\})$, where   $D_{2n}=\langle a,b|a^n=b^2=(ba)^2=1 \rangle$,   $r \in \mathbb{Z}_n^*$ such that  $r^4+r^3+r^2+r+1 \equiv 0 \pmod n$.
 \end{enumerate}

\end{theo}

The complete graph $\K_6$ is isomorphic to the graph $BC_3[\{\pm 1\},\{0,1,2\},\{\pm 1\}]$  in \cite[Theorem 1.1]{AHK-2015}. Remark 1.2 of  \cite{AHK-2015} indicates that  $\K_{6,6}-6\K_2\cong BC_6[\varnothing,\{0,1,2,3,4\},\varnothing]\cong BC_6[\{\pm 1,3\},$ $\{0,2\},\{\pm 1,3\}]\cong BC_6[\{\pm 1\},$ $\{0,2,4\},\{\pm 1\}]$, and
$ B(\PG(2,4))\cong BC_{21}[\varnothing,$ $\{0,1,4,14,16\},\varnothing]$.
It can be  checked  by Magma \cite{Magma-1997}  that the Clebsch graph is the graph $BC_8[\{\pm 1,4\},\{0,2\},\{\pm 3,4\}]$ in \cite[Theorem 1.1]{AHK-2015}.
The graph $\K_{5,5}$ is the Cayley graph $\Cay(D_{10},\{b,ba,ba^{2},ba^{3},ba^{4}\})$ in Theorem \ref{atbic-val5-1} (h).

\medskip

The graphs $\K_6$, $\K_{6,6}-6\K_2$,  $B(\PG(2,4))$ and the Clebsch graph are in Theorem \ref{bicirculant-reduction-th2}. Let $\Gamma=BC_{12}[\varnothing,\{0,1,2,4,9\},\varnothing]$ and  $\Gamma'=BC_{24}[\varnothing,\{0,1,3,11,20\},\varnothing]$. Then by \cite[Theorem 3.11,p.666]{AHK-2015},
$\Aut(\Gamma)$ has a normal subgroup $M\cong \mathbb{Z}_2$
 and $\Aut(\Gamma')$ has a normal subgroup $N\cong \mathbb{Z}_4$ such that
$\Gamma_M\cong \Gamma'_N\cong \K_{6,6}-6\K_2$.
It can be easily checked  using Magma \cite{Magma-1997}  that    the automorphism group of $BC_6[\{\pm 1\},\{0,1,5\},\{\pm 2\}]$ is $\PSL(2,5)\times S_2$.

Now we  use Lemma \ref{dihval35-1} to show how each valency 5 arc-transitive bicirculant gives rise to a graph in
Theorem \ref{bicirculant-reduction-th2}.

\begin{table}[h]\caption{ Pentavalent arc-transitive bicirculants }
\begin{tabular}{|c|c|c|c|c|}
\hline
 $\Gamma$ &  $\Gamma_{N}$ & $\Aut(\Gamma)$ &  $N$     \\
\hline
$\K_6$ & $\K_6$ &  $S_6$   &  $1$      \\
\hline
 $\K_{6,6}-6\K_2$ &   $\K_{6,6}-6\K_2$ &  $S_6\times S_2$   & $1$      \\
 \hline
$B(\PG(2,4))$ & $B(\PG(2,4))$  & $P\Gamma L(3,2):S_2$    & $1$     \\
\hline
Clebsch graph &  Clebsch graph & $\mathbb{Z}_2^4:S_5$ &  $1$          \\
 \hline
$BC_6[\{\pm 1\},\{0,1,5\},\{\pm 2\}]$ & $\K_6$  & $PSL(2,5)\times S_2$   &  $S_2$     \\
 \hline
$BC_{12}[\varnothing,\{0,1,2,4,9\},\varnothing]$ & $\K_{6,6}-6\K_2$  &  $|\Aut(\Gamma)|=480$  &  $\mathbb{Z}_2$     \\
 \hline
$BC_{24}[\varnothing,\{0,1,3,11,20\},\varnothing]$  &  $\K_{6,6}-6\K_2$  &  $|\Aut(\Gamma)|=960$  &  $\mathbb{Z}_4$     \\
\hline
$\Cay(D_{2n},\{b,ba,ba^{r+1},ba^{r^2+r+1},ba^{r^3+r^2+r+1}\})$, & $G(2p,5)$,   &   & $\mathbb{Z}_{n/p}$          \\
 as in Theorem \ref{atbic-val5-1}& prime $p|n$, $5|p-1$ &   &          \\
\hline
\end{tabular}
\label{table-val5}
\end{table}

\begin{prop}\label{bicirculant-val5}
Let $\Gamma$ be a finite connected arc-transitive  bicirculant of valency $5$.
Then there exists a normal subgroup $N$ of $\Aut(\Gamma)$ such that $\Gamma_N$ is a graph in Theorem \ref{bicirculant-reduction-th2}.
Moreover,  $\Gamma$, $\Gamma_N$, $\Aut(\Gamma)$ (if known), and $N$
are as in Table \ref{table-val5}.

\end{prop}

\subsection{Arc-transitive bicirculants of valency 4}
We conclude with a brief discussion of the valency four case.  By Theorem \ref{bicirculant-reduction-th2}, such a graph is either a cover of one of the graphs of valency four listed or a 2-cover of one of the graphs of valency two listed. Note that a connected graph of valency two is a cycle and appears in Theorem \ref{bicirculant-reduction-th2} as $\Cay(p,2)$ for $p$ odd, or as $\K_{2,2}$.  The graphs of valency four listed in Theorem \ref{bicirculant-reduction-th2}  are
\begin{enumerate}[(a)]
\item $\K_{4,4}$;
\item $\K_5$;
\item $\K_{5,5}-5\K_2$;
\item $\K_{3[2]}$;
\item $G(2p,4)$ with $p\equiv 1\pmod 4$;
\item $B(\PG(2,3))$;
\item $B'(\PG(2,2)$;
\item $\Cay(p,4)$ with $p\equiv 1\pmod 4$.
\end{enumerate}

All finite arc-transitive bicirculants of valency 4 were classified in \cite{KKMW-2012}.


\begin{thebibliography}{hhhh}


\bibitem{ACMX-1996}
B. Alspach, M. Conder, D. Maru$\check{{\rm s}}$i$\check{{\rm c}}$ and M. Y. Xu,  A classification of
2-arc transitive circulants, {\it J. Algebraic Combin.} {\bf 5}
(1996), 83--86.


\bibitem{AHK-2015}
I. Anton\v ci\v c, A. Hujdurovi\v c and K. Kutnar,  A classification of
pentavalent arc-transitive bicirculants, {\it J. Algebraic Combin.} {\bf 41}
(2015), 643--668.

\bibitem{AHKOS-2015}
A. Arroyo, I. Hubard,  K. Kutnar,  E. O'Reilly and P. \v Sparl,   Classification of
symmetric Taba\v cjn graphs, {\it Graphs  Combin.} {\bf 31}
(2015), 1137--1153.






\bibitem{Magma-1997}
W. Bosma, C. Cannon, and C. Playoust, The MAGMA algebra system I: The user
language, {\it J. Symbolic Comput.} {\bf 24} (1997), 235--265.



\bibitem{Cameron-1}
P. J. Cameron,  {\it Permutation Groups}, volume 45 of London Mathematical
Society Student Texts, Cambridge University Press, Cambridge,
(1999).

\bibitem{Chao-1971}
C. Y. Chao,  On the classification of symmetric graphs with a prime
number of vertices, {\it Trans. Amer.  Math. Soc.} {\bf 157} (1971),
247--256.



\bibitem{CO-1987}
Y. Cheng and J. Oxley, On weakly symmetric graphs of order twice a
prime, {\it J. Combin. Theory Ser. B} {\bf 42} (1987), 196--211.


\bibitem{DGLPP}
A. Devillers, M. Giudici, C. H. Li, G. Pearce and C. E. Praeger, On
imprimitive rank 3 permutation groups, {\it J. London  Math. Soc.}
{\bf (2)85} (2012), 649--669.










\bibitem{DM-1}
J. D. Dixon and B. Mortimer,  {\it Permutation groups}, Springer, New
York, (1996).






\bibitem{FGW-1971}
R. Frucht, J. E. Graver, M. E. Watkins, The groups of the generalized Petersen graphs, {\it Proc. Camb. Philos. Soc.} {\bf 70} (1971),
211--218.



\bibitem{GLP-1}
M. Giudici, C. H. Li and C. E. Praeger, Analysing finite locally
$s$-arc transitive graphs, {\it Trans. Amer. Math. Soc.} {\bf 356}
 (2003), 291--317.


\bibitem{Godsil-1981}
C. D. Godsil,  On the full automorphism group of a graph, {\it
Combinatorica} {\bf 1} (1981), 243--256.


\bibitem{GLP}
C. D. Godsil, R. A. Liebler and C. E. Praeger,  Antipodal distance
transitive covers of complete graphs, {\it European J. Combin.} {\bf
19} (1998), 455--478.






\bibitem{HPZ-1}
B. Horvat, T. Pisanski and A. \v Zitnik,  Isomorphism checking of I-graphs,
{\it Graphs Combin.} {\bf 28(6)} (2012), 823--830.










\bibitem{JMSV}
R. Jajcay, S. Miklavi\v c, P. \v Sparl and G. Vasiljevi\'c, On certain edge-transitive bicirculants, 	  {\it Electron. J. Comb.} \textbf{26(2)}, \#P2.6 (2019).






\bibitem{Jones-2002}
G. Jones, Cyclic regular subgroups of primitive permutation groups, {\it J. Group Theory} {\bf 5(4)} (2002), 403--407.





\bibitem{Kovacs-2004}
I. Kov\'acs,  Classifying arc-transitive circulants, {\it J.
Algebraic Combin. } {\bf 20} (2004), 353--358.



\bibitem{KKM-2010}
I. Kov\'acs, K. Kutnar and D. Maru\v si\v c,  Classification of edge-transitive rose window graphs, {\it J.
Graph Theory } {\bf 65} (2010), 216--231.


\bibitem{KKM}
I. Kov\'acs, B. Kuzman and A. Malni\v c,  On non-normal arc-transitive 4-valent dihedrants, {\it Acta Math.
Sinica, English Series } {\bf 26} (2010), 1485--1498.

\bibitem{KKMW-2012}
I. Kov\'acs, B. Kuzman,  A. Malni\v c and S. Wilson,  Characterization of edge-transitive 4-valent bicirculants, {\it J. Graph Theory } {\bf 69} (2012), 441--463.





\bibitem{RJ-1992}
M. J. de Resmini and D. Jungnickel, Strongly regular semi-Cayley graphs, {\it J. Algebr. Combin.} {\bf 1} (1992) 171--195.

\bibitem{LM-1993}
K. H. Leung, S. L. Ma, Partial difference triples, {\it J. Algebr. Combin.} {\bf 2} (1993) 397--409.


\bibitem{LCH-abelianregular-2003}
C. H. Li, The finite primitive permutation groups containing an abelian regular subgroup, {\it Proc. London Math. Soc.} {\bf 87} (2003),
725--748.

\bibitem{LCH-circulant-2005}
C. H. Li, Permutation groups with a cyclic regular subgroup and arc
transitive circulants, {\it J. Algebraic Combin.} {\bf 21} (2005),
131--136.



\bibitem{LP-circulant-2012}
C. H. Li and C. E. Praeger, On finite permutation groups with a transitive cyclic  subgroup, {\it J. Algebra} {\bf 349} (2012), 117--127.











\bibitem{MMSF-2007}
A. Malni$\check{{\rm c}}$, D. Maru$\check{{\rm s}}$i$\check{{\rm c}}$, P. $\check{{\rm S}}$parl and B. Frelih,   Symmetry structure of bicirculants,
{\it Discrete Math.} {\bf 307} (2007), 409--414.

\bibitem{MP-2000}
D. Maru$\check{{\rm s}}$i$\check{{\rm c}}$ and T. Pisanski,   Symmetries of hexagonal molecular graphs on the torus, {\it Croat. Chem. Acta} {\bf 73} (2000), 969--981.


\bibitem{Muller-2018}
P. M\"{u}ller, Permutation groups with a cyclic two-orbits subgroup and monodromy groups
of Laurent polynomials,  {\it Ann. Sc. Norm. Super. Pisa CI. Sci.} {\bf 12} (2013), 369--438.

%



\bibitem{P-2007}
T. Pisanski,  A classification of cubic bicirculants,
{\it Discrete Math.} {\bf 307} (2007), 567--578.




\bibitem{Praeger-1993-onanscott}
C. E. Praeger,  An O'Nan Scott theorem for finite quasiprimitive
permutation groups and an application to 2-arc transitive graphs,
{\it J. London Math. Soc.} {\bf 47 (2)} (1993), 227--239.

\bibitem{Praeger-2}
C. E. Praeger,  Finite Transitive Permutation Groups and Finite
Vertex-Transitive graphs,   Graph Symmetry: Algebraic Methods and
Applications, {\it NATO Adv. Sci. Inst. Ser.C Math. Phys. Sci.} {\bf
497} (1997), 277--318.



\bibitem{Praeger-2003-biq}
C. E. Praeger, Finite transitive permutation groups and  bipartite vertex-transitive graphs, {\it Illinois J. Mathematics} {\bf 47} (2003) 461--475.


















\bibitem{Wielandt-book}
H. Wielandt, {\it Finite Permutation Groups}, New York: Academic Press
(1964).

\bibitem{Wilson}
S. Wilson, Rose window graphs, \emph{Ars Math. Contemp.}, \textbf{1} (2008), 7--19.


\bibitem{ZF-bicay-2014}
J. X. Zhou and Y. Q. Feng,  Cubic bi-Cayley graphs over abelian groups,
{\it European J. Combin.}, {\bf 36} (2014), 679--693.

\bibitem{ZF-bicir-2016}
J. X. Zhou and Y. Q. Feng,  The automorphisms of bi-Cayley graphs,
{\it J. Combin. Theory Ser. B}, {\bf
116} (2016), 504--532.



\end{thebibliography}
\end{document}